\documentclass[12pt]{article}
\usepackage{amsmath,amssymb,amsthm}
\newtheorem{thm}{Theorem}[section]

 \newtheorem{cor}{Corollary}[section]
 \newtheorem{lem}{Lemma}[section]
 \newtheorem{prop}{Proposition}[section]
 \newtheorem{defn}{Definition}[section]
\newtheorem{rem}{Remark}[section]
\newtheorem{claim}{Claim}[section]

\begin{document}
\begin{center}
{\large{\bf The optimal decay estimates for the Euler-Poisson two-fluid system}}
\end{center}
\begin{center}
\footnotesize{Jiang Xu}\\[2ex]
\footnotesize{Department of Mathematics, \\ Nanjing
University of Aeronautics and Astronautics, \\
Nanjing 211106, P.R.China,}\\
\footnotesize{jiangxu\underline{ }79@nuaa.edu.cn}\\

\vspace{3mm}

\footnotesize{Faculty of Mathematics, \\ Kyushu University, Fukuoka 812-8581, Japan}\\

\vspace{5mm}

\footnotesize{Shuichi Kawashima}\\[2ex]
\footnotesize{Faculty of Mathematics, \\ Kyushu University, Fukuoka 812-8581, Japan,}\\
\footnotesize{kawashim@math.kyushu-u.ac.jp}\\
\end{center}
\vspace{6mm}

\begin{abstract}
This work is devoted to the optimal decay problem for the Euler-Poisson two-fluid system,
which is a classical hydrodynamic model arising in semiconductor sciences.
 By exploring
the influence of the electronic field on the dissipative structure, it is first revealed that
the \textit{irrotationality} plays a key role such that the two-fluid system
has the same dissipative structure as generally hyperbolic systems satisfying the Shizuta-Kawashima condition.
The fact inspires us to give a new decay framework which pays less attention on the traditional spectral analysis.
Furthermore, various decay estimates of solution and its derivatives of fractional order on the framework of Besov spaces are obtained by time-weighted energy approaches in terms of low-frequency and high-frequency decompositions. As direct consequences, the optimal decay rates of $L^{p}(\mathbb{R}^{3})$-$L^{2}(\mathbb{R}^{3})(1\leq p<2)$ type for the Euler-Poisson two-fluid system are also shown.
\end{abstract}


\section{Introduction}
\par
Consider an un-magnetized plasma consisting of
electrons with (scaled) mass $m_{e}$ and charge $q_{e}=-1$ and of a
single species of ions with mass $m_{i}$ and charge $q_{i}=+1$. We
denote by $n_{e}=n_{e}(t,x), u_{e}$ ($n_{i}, u_{i}$, respectively)
the density and current density of electrons (ions, respectively)
and by $\mathit\Phi=\mathit\Phi(t,x)$ the electrostatic potential.
By some appropriate re-scaling,  the hydrodynamic model reads as
(see for example, \cite{MRS})
\begin{equation}
\left\{
\begin{array}{l}\frac{\partial}{\partial t}n_a+\nabla\cdot(n_au_a)=0,\\
 m_a\frac{\partial}{\partial t}(n_au_a)+m_a\nabla\cdot(n_au_a\otimes u_a)
+\nabla p_a(n_a)\\
\hspace{50mm}=-q_an_a\nabla\mathit\Phi-m_a\frac{n_au_a}{\tau_a},
 \\
 \lambda^2\Delta\mathit\Phi=n_{e}-n_{i},\  \lim_{|x|\rightarrow +\infty}\mathit\Phi(t,x)=0,
 \end{array} \right.\label{R-E1}
\end{equation}
with $a=e,i$ and $(t,x)\in[0,+\infty)\times\mathbb{R}^{3}$, where
$m_{e},m_{i}$ are the (scaled) electron and ion mass,
$\tau_{e},\tau_{i}>0$ are the momentum relaxation times of electrons
and ion, respectively, and $\lambda>0$ is the Debye length. In this paper, we set these physical constants to be one.
The pressure $p_{a}(a=e,i)$ is a smooth function satisfying
\begin{eqnarray*}
p'_{a}(n_{a})>0 \ \ \mbox{for all}\ \ n_{a}>0.
\end{eqnarray*}

System (\ref{R-E1}) is supplemented by initial conditions for
$n_{a}$ and $u_a \;(a=e,i)$:
\begin{eqnarray}
n_{a}(x,0)=n_{a0}(x),\ \ \ u_{a}(x,0)=u_{a0}(x), \label{R-E2}
\end{eqnarray}

As well known (\cite{MRS}), the time evolution
of the distributes of electrons and positively charged ions in a plasma is well described by the semiclassical Boltzmann-Poisson equations. Unfortunately,
dealing with the kinetic equations remains too expensive from a computational point of view. Consequently,
it is possible to derive some simpler fluid dynamical equations for macroscopic quantities
like density, velocity and energy density, which represents a comprise between physical accuracy
and reduction of computational cost. System (\ref{R-E1}) reduces to the one-fluid Euler-Poisson equations, if the time evolution of electrons is considered only.

\subsection{Known results}
So far there are various topics in mathematical analysis for (\ref{R-E1})-(\ref{R-E2}), such as
the well-posedness of steady state solutions, global existence and large time behavior of solutions and singular limit problems, etc.,
the reader is referred to \cite{A,ACJP,AJ,CG,CW,De,FXZ,G,GHL,HMW1,HMW,HMWY,HZ,JP1,JP2,L,LNX,LY,N,PX,W1,X1,XZ,Y,ZH} and references therein. For brevity, let us only review the global existence and decay estimates of classical solutions for the one-fluid case.
Luo, Natalini and Xin \cite{LNX} first established
the global \textit{exponential} stability of classical solutions near the constant equilibrium in one dimension space.
Guo \cite{G} investigated the irrotational case $(\nabla\times u=0)$ and smooth irrotational solutions are constructed based on the Klein-Gordon
effect, which decay to the equilibrium state uniformly as $(1+t)^{-p}(1<p<3/2)$. Hsiao, Markowich and Wang \cite{HMW1}
dealt with the multidimensional unbounded domain problem $(n=2,3)$ without any geometrical assumptions. Subsequently, Fang,
the first author and Zhang \cite{FXZ,X1}, by performing low-frequency and high-frequency
decomposition methods, established the global exponential stability and diffusive relaxation-limit
of classical solutions on the framework of spatially critical Besov spaces.

The two-fluid equations (\ref{R-E1}) have also received more and more attention.
In one space dimension, Natalini \cite{N}, Wang \cite{W1}, Hsiao and Zhang \cite{HZ} established the global entropy
weak solutions by using the compensated compactness theory, respectively. Zhu and Hattori \cite{ZH} proved the stability of
steady-state solutions for a recombined two fluid Euler-Poisson equations.
Gasser, Hsiao and Li \cite{GHL} investigated the nonlinear diffusive phenomena
of hyperbolic waves. Subsequently, Huang, Mei, Wang and Yang \cite{HMWY} showed the convergence
of the original solution to the diffusion wave with optimal
convergence rates. J\"{u}ngel and Peng \cite{JP1,JP2} justified the
zero-relaxation-time limits based on appropriate compactness arguments.

In the multi-dimensional case, Lattanzio \cite{L} considered the
relaxation limit in a compactness framework for non-smooth solutions under
the assumption that the $L^\infty$-solutions exist in a $\tau$-independent
time interval. The zero-electron-mass limit of (\ref{R-E1})-(\ref{R-E2}) with in the case of ``well-prepared" initial data
was studied by Al\`{\i}, Chen, J\"{u}ngel and Peng \cite{ACJP}. The first author and Zhang \cite{XZ}
developed the frequency-localization Strichartz estimates and
investigated the case of ``ill-prepared" initial data. Huang, Mei and Wang
\cite{HMW} proved the stability of planar diffusion waves. Al\`{\i}
and J\"{u}ngel introduced a technical condition (see \cite{AJ}) that
the electric field $E$ can be divided into two parts
and each part was generated by carriers separately, and
studied the global exponential stability of smooth solutions
to the Cauchy problem. Actually, the condition
reduces the nonlinear interaction between two carriers heavily so that solutions behave as the case of one-fluid.
Recently, Peng and the first author \cite{PX} removed the technical condition and captured the dissipation for $n_{e}-n_{i}$.
Furthermore, global classical solutions
was constructed in the critical Besov spaces. However, the corresponding decay problem in whole space was left open in \cite{PX}.

Based on the decay framework in \cite{UKS}, the second author \cite{Ka} studied generally hyperbolic-parabolic composite systems satisfying the Shizuta-Kawashima's condition and obtained the optimal decay estimates in $H^{l}(\mathbb{R}^{n})\cap L^{1}(\mathbb{R}^{n})(l>2+n/2, l\in\mathbb{Z})$. This effort
has been developed great, for instance, by Hoff and Zumbrun \cite{HZ2} for compressible Navier-Stokes equations, where they
employed the elaborate spectral analysis on the Green's matrix. Li and Yang \cite{LY} first considered the two-fluid equations (\ref{R-E1})-(\ref{R-E2})
by virtue of the spectral analysis  and
showed that the densities converge to its equilibrium state at the rates $(1+t)^{-3/4}$ in the $L^2$-norm and the velocities as well as the electronic field decay at the rates $(1+t)^{-1/4}$ in the $L^2$-norm, as the initial data
$(n_{a0}-\bar{n},u_{a0})\in H^{l}(\mathbb{R}^3)\cap L^{1}(\mathbb{R}^3) (l\geq4)$. To the best of our knowledge, such decay rates are far away from the optimal case, since the decay estimates for velocities and the electronic field are more slowly than that of the standard heat kernel.

 Very recently, based on the work \cite{XK1}, the authors introduced a decay framework for general dissipative hyperbolic system and hyperbolic-parabolic composite system satisfying the Shizuta-Kawashima condition, which allows to pay less attention on the traditional spectral analysis, if the initial data belong to $B^{s_{c}}_{2,1}(\mathbb{R}^n)\cap \dot{B}^{-s}_{2,\infty}(\mathbb{R}^n)(s_{c}:=1+n/2,\ s\in (0,n/2])$.
The new framework can be regarded as the great improvement of \cite{Ka,UKS}, since
$L^{1}(\mathbb{R}^{n})\hookrightarrow\dot{B}^{0}_{1,\infty}(\mathbb{R}^{n})\hookrightarrow\dot{B}^{-n/2}_{2,\infty}(\mathbb{R}^{n})$ and $H^{l}(\mathbb{R}^n)(l>2+n/2, l\in\mathbb{Z})\hookrightarrow B^{s_{c}}_{2,1}(\mathbb{R}^n)$. The interested reader is referred to \cite{XK2}.
The main aim of this paper is to answer the optimal decay for the Euler-Poisson two-fluid system by exploring the influence of the coupled electronic field on the dissipative structure, which is an interesting problem left.

\subsection{Reformulation and main results}
It is convenient to reformulate the two-fluid system (\ref{R-E1}) around the equilibrium state $(\bar{n},0,\bar{n},0,0)$.
Without loss of generality, let us assume that
$\bar{n}=1$ and $p'(\bar{n})=1$.  Denote

$$\sigma_{a}=n_{a}-1,\ \ \ h(\sigma_{a})=\frac{p'(n_{a})}{n_{a}}-1,\ \ \ \ a=e,i. $$
Then, we have
\begin{equation}
\left\{
\begin{array}{l}
\partial_{t} \sigma_{e}+\mathrm{div}u_{e}=-u_{e}\cdot\nabla\sigma_{e}-\sigma_{e}\mathrm{div}u_{e},\\
\partial_{t} u_{e}+\nabla\sigma_{e}+u_{e}-E=-u_{e}\cdot\nabla u_{e}-h(\sigma_{e})\nabla\sigma_{e},\\
\partial_{t} \sigma_{i}+\mathrm{div}u_{i}=-u_{i}\cdot\nabla\sigma_{i}-\sigma_{i}\mathrm{div}u_{i},\\
\partial_{t} u_{i}+\nabla\sigma_{i}+u_{i}+E=-u_{i}\cdot\nabla u_{i}-h(\sigma_{i})\nabla\sigma_{i},\\
\mathrm{div} E=\sigma_{e}-\sigma_{i}, \ \ \ E=\nabla\mathit\Phi,
 \end{array}
\right.\label{R-E3}
\end{equation}
with the initial data
\begin{eqnarray}
\sigma_{a}(x,0)=n_{a0}(x)-1,\ \ \ u_{a}(x,0)=u_{a0}(x), \ \ \  a=e,i. \label{R-E333}
\end{eqnarray}
The corresponding linearized system reads as
\begin{equation}
\left\{
\begin{array}{l}
\partial_{t} \sigma_{e}+\mathrm{div}u_{e}=0,\\
\partial_{t} u_{e}+\nabla\sigma_{e}+u_{e}-E=0,\\
\partial_{t} \sigma_{i}+\mathrm{div}u_{i}=0,\\
\partial_{t} u_{i}+\nabla\sigma_{i}+u_{i}+E=0,\\
\mathrm{div} E=\sigma_{e}-\sigma_{i}, \ \ \ E=\nabla\mathit\Phi.
 \end{array}
\right.\label{R-E4}
\end{equation}

In what follows, we explore the influence of $E$ and understand the dissipative structure of (\ref{R-E4}) in essential.
Set $$w=(\sigma_{e},u_{e},\sigma_{i},u_{i}),\ \ \ \ \tilde{w}=(w,E).$$
More concretely speaking, by using the energy method in Fourier spaces, we get
\begin{eqnarray}
\frac{1}{2}\frac{d}{dt}|\hat{\tilde{w}}|^2+|(\hat{u}_{e},\hat{u}_{i})|^2=0 \label{R-E5}
\end{eqnarray}
and
\begin{eqnarray}
&&\frac{1}{2}\frac{d}{dt}\mathrm{Im}\Big\langle\frac{|\xi|}{1+|\xi|^2}K(\xi)\hat{w},\hat{w}\Big\rangle+
\frac{|\xi|^2}{1+|\xi|^2}|\hat{w}|^2+\frac{1}{1+|\xi|^2}|\widehat{\mathrm{div}E}|^2\nonumber\\&\leq& C|(\hat{u}_{e},\hat{u}_{i})|^2,
\label{R-E6}
\end{eqnarray}
where $\hat{f}$ denotes the Fourier transform of the function $f$ and the matrix $K(\xi)$ is defined by Lemma \ref{lem3.1} in Sect.~\ref{sec:3}.

The fact $\mathrm{curl}E=0$ implies that $\xi\times\hat{E}=0$ which leads to $|\xi|^2|\hat{E}|^2\approx|\xi\cdot\hat{E}|^2$.
Then, (\ref{R-E6}) becomes into
\begin{eqnarray}
\frac{1}{2}\frac{d}{dt}\mathrm{Im}\Big\langle\frac{|\xi|}{1+|\xi|^2}K(\xi)\hat{w},\hat{w}\Big\rangle+
\frac{|\xi|^2}{1+|\xi|^2}|\hat{\tilde{w}}|^2\leq C|(\hat{u}_{e},\hat{u}_{i})|^2. \label{R-E7}
\end{eqnarray}

Therefore, the linearized system (\ref{R-E4}) admits a Lyapunov function of the form
\begin{eqnarray}
E(\hat{\tilde{w}})=\frac{1}{2}|\hat{\tilde{w}}|^2+\frac{\kappa}{2}\mathrm{Im}\Big\langle\frac{|\xi|}{1+|\xi|^2}K(\xi)\hat{w},\hat{w}\Big\rangle, \label{R-E8}
\end{eqnarray}
where $\kappa>0$ is a small constant. Then it is shown that
\begin{eqnarray}
\frac{d}{dt}E(\hat{\tilde{w}})+(1-\kappa C)|(\hat{u}_{e},\hat{u}_{i})|^2+\frac{\kappa|\xi|^2}{1+|\xi|^2}|\hat{\tilde{w}}|^2\leq0, \label{R-E9}
\end{eqnarray}
where we can choose $\kappa>0$ so small that $1-\kappa C\geq0$ and $E(\hat{w})\approx|\hat{\tilde{w}}|^2.$ Furthermore, there exists a constant $c_{0}>0$ such that
 \begin{eqnarray}
|\hat{\tilde{w}}|\leq |\hat{\tilde{w}}_{0}|e^{-c_{0}\eta(\xi)t}, \label{R-E999}
\end{eqnarray}
where $\eta(\xi):=|\xi|^2/(1+|\xi|^2)$.

\begin{rem}\label{rem1.1}
The dissipative structure (\ref{R-E999}) is just the same one as general dissipative systems studied in \cite{UKS}.
The above calculations reveal that the \textit{irrotationality} property of the electronic field $E$ plays a key role. Furthermore, we
develop the Littlewood-Paley pointwise energy estimates for (\ref{R-E4}) on the framework of Besov spaces, see Sect.~\ref{sec:3}.
\end{rem}

Let us sketch the technical obstruction of this paper. To obtain the optimal decay estimates for (\ref{R-E1})-(\ref{R-E2}),
the idea of time-weighted energy estimates which was first established by Matsumura \cite{Ma} is mainly used. Here, in virtue of frequency-localization
Duhamel principle, the time-weighted energy approach in terms of low frequency and high-frequency decomposition are well developed.
Additionally, there appears a difficulty arising from the coupled electronic field $E$ in order to obtain the 1/2 faster
decay rate for the non-degenerate quantities, say velocities. Indeed, we are unable to obtain the sharp decay estimates for velocities directly, since $E$ has no additional half rate. Here we involve some interesting observations on the information behind the equations. Precisely,
adding the two velocity equation in (\ref{R-E3}) to eliminate $E$, which inspire us to obtain the sharp time-weighted decay estimates for the sum of two velocities. To close the weighted energy inequality, it suffices to get the sharp estimates for the difference of two velocities.
Fortunately, it follows from the linearized system (\ref{R-E4}) that
\begin{equation}
\left\{
\begin{array}{l}
\partial_{t} \tilde{\sigma}+\mathrm{div}\tilde{u}=0\\
\partial_{t} \tilde{u}+\nabla\tilde{\sigma}+\tilde{u}=2E\\
\mathrm{div}E=\tilde{\sigma},\ \ E=\nabla\mathit\Phi
 \end{array}
\right. \label{R-E1000}
\end{equation}
with $\tilde{u}=u_{e}-u_{i}$ and $\tilde{\sigma}=\sigma_{e}-\sigma_{i},$ which exactly consists of a one-fluid Euler-Poisson equations.
As shown by \cite{FXZ,HMW1,LNX,X1}, the Euler-Poisson one-fluid system has the exponential stability of classical solutions.
Therefore, we can employ the high-frequency and low-frequency estimates for (\ref{R-E1000}) with the operator $\Delta_{q}\Lambda^{\ell}(q\geq-1, \ 0\leq\ell\leq s_{c}-2)$ and get the exponential decay for linearized solution $(\tilde{\sigma},\tilde{u},E)$. Finally,
the sharp decay estimates for the difference $\tilde{u}$ of velocities can follow from the frequency-localization
Duhamel principle. See the proofs of Lemmas \ref{lem4.4}-\ref{lem4.5} for details.

For the convenience of reader, let us first recall the global-in-time existence of solutions in spatially critical Besov spaces achieved in \cite{PX} ($s_{c}:=5/2$).
\begin{thm}\label{thm1.1}
Suppose that $(n_{a0}-1,u_{a0},E_{0})\in B^{s_{c}}_{2,1}(\mathbb{R}^3)$ where $E_{0}:=\nabla\Delta^{-1}(n_{e0}-n_{i0}).$
There exists a positive constant $\delta_{0}$ such that if
$$\|(n_{a0}-1,u_{a0},E_{0})\|_{B^{s_{c}}_{2,1}(\mathbb{R}^3)}\leq
\delta_{0}, \quad (a=e,i), $$
then system (\ref{R-E1})-(\ref{R-E2}) admits a
unique classical solution $$(n_{a},u_{a},E)\in
\mathcal{C}^{1}([0,\infty)\times \mathbb{R}^3)$$ satisfying
$$(n_{a}-1,u_{a},E) \in\widetilde{\mathcal{C}}(B^{s_{c}}_{2,1}(\mathbb{R}^3))\cap\widetilde{\mathcal{C}}^{1}(B^{s_{c}-1}_{2,1}(\mathbb{R}^3)), \quad (a=e,i). $$
Moreover, the following energy inequality holds
\begin{eqnarray*}
&&\|(n_{a}-1,u_{a},E)\|_{\widetilde{L}^{\infty}(B^{s_{c}}_{2,1}(\mathbb{R}^3))}\nonumber\\&&
+\mu_{0}\Big\{\|(n_{e}-n_{i},u_{a},E)\|_{\widetilde{L}^{2}(B^{s_{c}}_{2,1}(\mathbb{R}^3))}
+\|\nabla n_{a}\|_{\widetilde{L}^{2}(B^{s_{c}-1}_{2,1}(\mathbb{R}^3))}\Big\}\nonumber\\
&\leq& C_{0}\|(n_{a0}-1,u_{a0},E_{0})\|_{B^{s_{c}}_{2,1}(\mathbb{R}^3)}, \quad (a=e,i),
\end{eqnarray*}
where $\mu_{0}$ and $C_{0}$ are two positive constants.
\end{thm}

\begin{rem}\label{rem1.2}
In the periodic domain $\mathbb{T}^3$, the dissipation rate from $(n_{e},n_{i})$  can be further available by using Poincar\'{e} inequality, which leads to
the exponential decay of classical solutions near to equilibrium, the interested reader is referred to \cite{PX} for details. However, the situation in whole space $\mathbb{R}^3$ is totally different. 
\end{rem}

In the following, we begin to state main results of this paper. Denote $\Lambda^{\alpha}f:=\mathcal{F}^{-1}|\xi|^{\alpha}\mathcal{F}f ( \alpha\in \mathbb{R})$.
\begin{thm}\label{thm2.4}
Let $(n_{a},u_{a},E)(t,x)$ be the global classical solution of Theorem \ref{thm1.1}. If further the initial data $(n_{a0}-1,u_{a0},E_{0})\in \dot{B}^{-s}_{2,\infty}(\mathbb{R}^{n})(0<s\leq 3/2)$ and
$$\mathcal{M}_{0}:=\|(n_{a0}-1,u_{a0},E_{0})\|_{B^{s_{c}}_{2,1}(\mathbb{R}^{3})\cap\dot{B}^{-s}_{2,\infty}(\mathbb{R}^{3})}$$
is sufficiently small. Then the classical solution $(n_{a},u_{a},E)(t,x)$  satisfies the following decay estimates
\begin{eqnarray}
\|\Lambda^{\ell}[(n_{a}-1,u_{a},E)]\|_{X_{1}(\mathbb{R}^{3})}\lesssim \mathcal{M}_{0}(1+t)^{-\frac{s+\ell}{2}} \label{R-E10}
\end{eqnarray}
for $0\leq\ell\leq s_{c}-1$, where $X_{1}:=B^{s_{c}-1-\ell}_{2,1}$ if $0\leq\ell<s_{c}-1$ and $X_{1}:=\dot{B}^{0}_{2,1}$ if $\ell=s_{c}-1$;
\begin{eqnarray}
\|\Lambda^{\ell}(u_{e},u_{i},n_{e}-n_{i})(t,\cdot)\|_{X_{2}(\mathbb{R}^{3})}\lesssim \mathcal{M}_{0}(1+t)^{-\frac{s+\ell+1}{2}} \label{R-E11}
\end{eqnarray}
for $0\leq\ell\leq s_{c}-2$, where $X_{2}:=B^{s_{c}-2-\ell}_{2,1}$ if $0\leq\ell<s_{c}-2$ and $X_{2}:=\dot{B}^{0}_{2,1}$ if $\ell=s_{c}-2$.
\end{thm}
Note that the $L^p(\mathbb{R}^{3})$ embedding property in Lemma \ref{lem8.5}, we obtain the optimal decay rates on the framework of Besov spaces.
\begin{thm}\label{thm2.5}
Let $(n_{a},u_{a},E)(t,x)$ be the global classical solution of Theorem \ref{thm1.1}. If further the initial data $(n_{a0}-1,u_{a0},E_{0})\in L^p(\mathbb{R}^{3})(1\leq p<2)$ and
$$\widetilde{\mathcal{M}}_{0}:=\|(n_{a0}-1,u_{a0},E_{0})\|_{B^{s_{c}}_{2,1}(\mathbb{R}^{3})\cap L^p(\mathbb{R}^{3})}$$
is sufficiently small. Then the classical solutions $(n_{a},u_{a},E)(t,x)$  satisfies the following optimal decay estimates
\begin{eqnarray}
\|\Lambda^{\ell}[(n_{a}-1,u_{a},E)]\|_{X_{1}(\mathbb{R}^{3})}\lesssim \widetilde{\mathcal{M}}_{0}(1+t)^{-\gamma_{p,2}-\frac{\ell}{2}} \label{R-E12}
\end{eqnarray}
for $0\leq\ell\leq s_{c}-1$,
and
\begin{eqnarray}
\|\Lambda^{\ell}(u_{e},u_{i},n_{e}-n_{i})(t,\cdot)\|_{X_{2}(\mathbb{R}^{3})}\lesssim \widetilde{\mathcal{M}}_{0}(1+t)^{-\gamma_{p,2}-\frac{\ell+1}{2}} \label{R-E13}
\end{eqnarray}
for $0\leq\ell\leq s_{c}-2$, where $X_{1}$ and $X_{2}$ are the same space notations as in Theorem \ref{thm2.4}. We denote by $\gamma_{p,2}:=\frac{3}{2}(\frac{1}{p}-\frac{1}{2})$ the $L^p(\mathbb{R}^{3})$-$L^2(\mathbb{R}^{3})$ decay rates for the heat kernel.
\end{thm}

\begin{rem}\label{rem1.3}
Let us mention that Theorems \ref{thm2.4}-\ref{thm2.5} exhibit the various decay rates of solution and its derivatives of fractional order. The harmonic analysis allows to reduce significantly the regularity requirements on the initial data in comparison with \cite{LY}.
It is worth noting that the derivative index $\ell$ can take values in the interval, for example, $[0,s_{c}-1]$ rather than nonnegative integers only. Additionally, the decay of the non-degenerate part $(u_{e},u_{i},n_{e}-n_{i})$ of solution is faster at half rate among all the components of solutions.
\end{rem}

As an immediate consequence of Theorems \ref{thm2.4}-\ref{thm2.5}, the optimal decay rates in the usual $L^2(\mathbb{R}^{3})$ space are available.
\begin{cor}\label{cor1.1}
Let $(n_{a},u_{a},E)(t,x)$ be the global classical solutions of Theorem \ref{thm1.1}.
\begin{itemize}
\item [(i)]  If $\mathcal{M}_{0}$ is sufficiently small, then
\begin{eqnarray}
\|\Lambda^{\ell}(n_{a}-1,u_{a},E)\|_{L^2(\mathbb{R}^{3})}\lesssim \mathcal{M}_{0}(1+t)^{-\frac{\ell+s}{2}},  \  0\leq\ell\leq s_{c}-1; \label{R-E14}
\end{eqnarray}
\begin{eqnarray}
\hspace{-10mm}\|\Lambda^{\ell}(u_{e},u_{i},n_{e}-n_{i})(t,\cdot)\|_{L^2(\mathbb{R}^{3})}\lesssim  \mathcal{M}_{0}(1+t)^{-\frac{\ell+s+1}{2}}, \  0\leq\ell\leq s_{c}-2. \label{R-E15}
\end{eqnarray}

\item [(ii)] If $\widetilde{\mathcal{M}}_{0}$ is sufficiently small, then
\begin{eqnarray}
\|\Lambda^{\ell}(n_{a}-1,u_{a},E)\|_{L^2(\mathbb{R}^{3})}\lesssim \widetilde{\mathcal{M}}_{0}(1+t)^{-\gamma_{p,2}-\frac{\ell}{2}}, \ 0\leq\ell\leq s_{c}-1; \label{R-E16}
\end{eqnarray}
\begin{eqnarray}
\hspace{-10mm}\|\Lambda^{\ell}(u_{e},u_{i},n_{e}-n_{i})(t,\cdot)\|_{L^2(\mathbb{R}^{3})}\lesssim \widetilde{\mathcal{M}}_{0}(1+t)^{-\gamma_{p,2}-\frac{\ell+1}{2}},\  0\leq\ell\leq s_{c}-2. \label{R-E17}
\end{eqnarray}
\end{itemize}
\end{cor}

\begin{rem}\label{rem1.4}
Taking $p=1$ in Corollary \ref{cor1.1}, we deduce the following important decay rates for the Euler-Poisson two-fluid system (\ref{R-E1})-(\ref{R-E2}):
$$\|(n_{e}-1, n_{i}-1,E)\|_{L^2}\lesssim (1+t)^{-\frac{3}{4}}, \ \ \ \|(\nabla n_{e}, \nabla n_{i},\nabla E)\|_{L^2}\lesssim (1+t)^{-\frac{5}{4}},$$
$$\|(u_{e},u_{i},n_{e}-n_{i})\|_{L^2}\lesssim (1+t)^{-\frac{5}{4}},$$
which improve those decay results in \cite{LY} on the framework of spatially Besov spaces of relatively weaker regularity.
\end{rem}

The paper is organized as follows. In Sect.~\ref{sec:2}, we review the Littlewood-Paley decomposition theory and present
the definition of Besov spaces as well as some useful inequalities in  Besov spaces. Sect.~\ref{sec:3} is devoted to develop the L-P pointwise energy estimates for the linearized system (\ref{R-E4}) and
deduce the decay estimates on the framework of spatially Besov spaces. In Sect.~\ref{sec:4}, we perform the modified time-weighted energy approach in terms of the low-frequency and high-frequency decomposition to obtain decay estimates for (\ref{R-E1})-(\ref{R-E2}). The paper will be end with an Appendix (Sect.~\ref{sec:5}), where we present some interpolation inequalities which are used in Sect.~\ref{sec:3} and Sect.~\ref{sec:4}.

\section{Preliminary}\setcounter{equation}{0}\label{sec:2}
Throughout the paper, we present some notations. Denote by
$\langle\cdot,\cdot\rangle$ to the standard inner product in $\mathbb{C}^{3}$. $f\lesssim g$ denotes $f\leq Cg$, where $C>0$
is a generic constant. $f\thickapprox g$ means $f\lesssim g$ and
$g\lesssim f$.  Denote by $\mathcal{C}([0,T],X)$ (resp., $\mathcal{C}^{1}([0,T],X)$)
the space of continuous (resp., continuously differentiable)
functions on $[0,T]$ with values in a Banach space $X$. For
simplicity, the notation $\|(f,g)\|_{X}$ means $\|f\|_{X}+\|g\|_{X}$ with $f,g\in X$.

The proofs of most of the results presented require a
dyadic decomposition of Fourier variables, so we recall briefly the
Littlewood-Paley decomposition and Besov spaces in $\mathbb{R}^{n}$. The reader also refers to \cite{BCD}
for more details.

Let us start with the Fourier transform. The Fourier transform $\hat{f}$ (or $\mathcal{F}f$)
of a $L^1$-function $f$ is given by
$$\mathcal{F}f=\int_{\mathbb{R}^{n}}f(x)e^{-2\pi x\cdot\xi}dx.$$ More
generally, the Fourier transform of a tempered distribution $f\in\mathcal{S}'$ is defined by
the dual argument in the standard way.

 Choose
$\phi_{0}\in \mathcal{S}$ such that $\phi_{0}$ is even,
$$\mathrm{supp}\phi_{0}:=A_{0}=\Big\{\xi\in\mathbb{R}^{n}:\frac{3}{4}\leq|\xi|\leq\frac{8}{3}\Big\},\  \mbox{and}\ \ \phi_{0}>0\ \ \mbox{on}\ \ A_{0}.$$
Set $A_{q}=2^{q}A_{0}$ for $q\in\mathbb{Z}$. Furthermore, we define
$$\phi_{q}(\xi)=\phi_{0}(2^{-q}\xi)$$ and define $\Phi_{q}\in
\mathcal{S}$ by
$$\mathcal{F}\Phi_{q}(\xi)=\frac{\phi_{q}(\xi)}{\sum_{q\in \mathbb{Z}}\phi_{q}(\xi)}.$$
It follows that both $\mathcal{F}\Phi_{q}(\xi)$ and $\Phi_{q}$ are
even and satisfy the following properties:
$$\mathcal{F}\Phi_{q}(\xi)=\mathcal{F}\Phi_{0}(2^{-q}\xi),\ \ \ \mathrm{supp}\ \mathcal{F}\Phi_{q}(\xi)\subset A_{q},\ \ \ \Phi_{q}(x)=2^{qn}\Phi_{0}(2^{q}x)$$
and
$$\sum_{q=-\infty}^{\infty}\mathcal{F}\Phi_{q}(\xi)=\begin{cases}1,\ \ \ \mbox{if}\ \ \xi\in\mathbb{R}^{n}\setminus \{0\},
\\ 0, \ \ \ \mbox{if}\ \ \xi=0.\end{cases}
$$

Let $\mathcal{P}$ be the class of all polynomials of $\mathbb{R}^{n}$ and denote by $\mathcal{S}'_{0}:=\mathcal{S}/\mathcal{P}$ the tempered
distributions modulo polynomials. As a consequence, for any $f\in \mathcal{S}'_{0},$ we have
$$\sum_{q=-\infty}^{\infty}\Phi_{q}\ast f=f.$$

Next, we give the definition of homogeneous Besov spaces. To do this,
 we set
$$\dot{\Delta}_{q}f=\Phi_{q}\ast f,\ \ \ \ q=0,\pm1,\pm2,...$$

\begin{defn}\label{defn3.1}
For $s\in \mathbb{R}$ and $1\leq p,r\leq\infty,$ the homogeneous
Besov spaces $\dot{B}^{s}_{p,r}$ is defined by
$$\dot{B}^{s}_{p,r}=\{f\in \mathcal{S}'_{0}:\|f\|_{\dot{B}^{s}_{p,r}}<\infty\},$$
where
$$\|f\|_{\dot{B}^{s}_{p,r}}
=\begin{cases}\Big(\sum_{q\in\mathbb{Z}}(2^{qs}\|\dot{\Delta}_{q}f\|_{L^p})^{r}\Big)^{1/r},\
\ r<\infty, \\ \sup_{q\in\mathbb{Z}}
2^{qs}\|\dot{\Delta}_{q}f\|_{L^p},\ \ r=\infty.\end{cases}
$$\end{defn}

To define the inhomogeneous Besov spaces, we set $\Psi\in
\mathcal{C}_{0}^{\infty}(\mathbb{R}^{n})$ be even and satisfy
$$\mathcal{F}\Psi(\xi)=1-\sum_{q=0}^{\infty}\mathcal{F}\Phi_{q}(\xi).$$
It is clear that for any $f\in S'_{0}$, yields
$$\Psi*f+\sum_{q=0}^{\infty}\Phi_{q}\ast f=f.$$
We further set
$$\Delta_{q}f=\begin{cases}0,\ \ \ \ \ \ \ \, \ j\leq-2,\\
\Psi*f,\ \ \ j=-1,\cr \Phi_{q}\ast f, \ \ j=0,1,2,...,\end{cases}$$
which leads to the definition of inhomogeneous Besov spaces.

\begin{defn}\label{defn3.2}
For $s\in \mathbb{R}$ and $1\leq p,r\leq\infty,$ the inhomogeneous
Besov spaces $B^{s}_{p,r}$ is defined by
$$B^{s}_{p,r}=\{f\in \mathcal{S}':\|f\|_{B^{s}_{p,r}}<\infty\},$$
where
$$\|f\|_{B^{s}_{p,r}}
=\begin{cases}\Big(\sum_{q=-1}^{\infty}(2^{qs}\|\Delta_{q}f\|_{L^p})^{r}\Big)^{1/r},\
\ r<\infty, \\ \sup_{q\geq-1} 2^{qs}\|\Delta_{q}f\|_{L^p},\ \
r=\infty.\end{cases}$$
\end{defn}

For convenience of reader, we present some useful facts as follows. The first one is the improved
Bernstein inequality, see, e.g., \cite{W}.

\begin{lem}\label{lem2.1}
Let $0<R_{1}<R_{2}$ and $1\leq a\leq b\leq\infty$.
\begin{itemize}
\item [(i)] If $\mathrm{Supp}\mathcal{F}f\subset \{\xi\in \mathbb{R}^{n}: |\xi|\leq
R_{1}\lambda\}$, then
\begin{eqnarray*}
\|\Lambda^{\alpha}f\|_{L^{b}}
\lesssim \lambda^{\alpha+n(\frac{1}{a}-\frac{1}{b})}\|f\|_{L^{a}}, \ \  \mbox{for any}\ \  \alpha\geq0;
\end{eqnarray*}

\item [(ii)]If $\mathrm{Supp}\mathcal{F}f\subset \{\xi\in \mathbb{R}^{n}:
R_{1}\lambda\leq|\xi|\leq R_{2}\lambda\}$, then
\begin{eqnarray*}
\|\Lambda^{\alpha}f\|_{L^{a}}\approx\lambda^{\alpha}\|f\|_{L^{a}}, \ \  \mbox{for any}\ \ \alpha\in\mathbb{R}.
\end{eqnarray*}
\end{itemize}
\end{lem}
As a consequence of the above inequality, we have
$$\|\Lambda^{\alpha} f\|_{B^s_{p, r}}\lesssim\|f\|_{B^{s +\alpha}_{p,
r}} \ (\alpha\geq0); \ \ \ \|\Lambda^{\alpha} f\|_{\dot{B}^s_{p, r}}\approx \|f\|_{\dot{B}^{s+\alpha}_{p, r}} \ (\alpha\in\mathbb{R}).$$
Below are basic embedding properties in Besov spaces.
\begin{lem}\label{lem2.2} Let $s\in \mathbb{R}$ and $1\leq
p,r\leq\infty$. Then
\begin{itemize}
\item[(1)] $\dot{B}^{0}_{p,1}\hookrightarrow L^p\hookrightarrow\dot{B}^{0}_{p,\infty},\ \  \dot{B}^{0}_{p,1}\hookrightarrow B^{0}_{p,1}$;
\item[(2)]$B^{s}_{p,r}=L^{p}\cap \dot{B}^{s}_{p,r} (s>0);$
\item[(3)]$B^{s}_{p,r}\hookrightarrow
B^{\tilde{s}}_{p,\tilde{r}}$ whenever $\tilde{s}< s$ or  $\tilde{s}=s$ and $r\leq\tilde{r}$;
\item[(4)]$\dot{B}^{s}_{p,r}\hookrightarrow \dot{B}^{s-n(\frac{1}{p}-\frac{1}{\tilde{p}})}_{\tilde{p},r}
$ and $B^{s}_{p,r}\hookrightarrow
B^{s-n(\frac{1}{p}-\frac{1}{\tilde{p}})}_{\tilde{p},r}$ whenever $p\leq\tilde{p}$;
\item[(5)]$\dot{B}^{n/p}_{p,1}\hookrightarrow\mathcal{C}_{0},\ \ B^{n/p}_{p,1}\hookrightarrow\mathcal{C}_{0}(1\leq p<\infty),$
where $\mathcal{C}_{0}$ is the space of continuous bounded functions
which decay at infinity.
\end{itemize}
\end{lem}
Let us state the Moser-type product estimates, which plays an important role in the estimate of bilinear
terms.
\begin{prop}\label{prop2.1}
Let $s>0$ and $1\leq
p,r\leq\infty$. Then $\dot{B}^{s}_{p,r}\cap L^{\infty}$ is an algebra and
$$
\|fg\|_{\dot{B}^{s}_{p,r}}\lesssim \|f\|_{L^{\infty}}\|g\|_{\dot{B}^{s}_{p,r}}+\|g\|_{L^{\infty}}\|f\|_{\dot{B}^{s}_{p,r}}.
$$
Let $s_{1},s_{2}\leq n/p$ such that $s_{1}+s_{2}>n\max\{0,\frac{2}{p}-1\}. $  Then one has
$$\|fg\|_{\dot{B}^{s_{1}+s_{2}-n/p}_{p,1}}\lesssim \|f\|_{\dot{B}^{s_{1}}_{p,1}}\|g\|_{\dot{B}^{s_{2}}_{p,1}}.$$
\end{prop}
Additionally, we also state a result of continuity for the composition function.
\begin{prop}\label{prop2.2}
Let $s>0$, $1\leq p, r\leq \infty$ and $F'\in
W^{[s]+1,\infty}_{loc}(I;\mathbb{R})$. Assume that $v\in \dot{B}^{s}_{p,r}\cap
L^{\infty},$ then $F(v)\in \dot{B}^{s}_{p,r}$ and
$$\|F(v)\|_{\dot{B}^{s}_{p,r}}\lesssim
(1+\|v\|_{L^{\infty}})^{n}\|F'\|_{W^{[s]+1,\infty}(I)}\|v\|_{\dot{B}^{s}_{p,r}}.$$
\end{prop}

Finally, for completeness, we present the definition of inhomogeneous space-time Besov spaces to end this section, which is used in Theorem \ref{thm1.1},
see \cite{CL} or \cite{BCD} for more details.

\begin{defn}\label{defn2.3}
For $T>0, s\in\mathbb{R}, 1\leq r,\theta\leq\infty$, the
inhomogeneous mixed time-space Besov spaces
$\widetilde{L}^{\theta}_{T}(B^{s}_{p,r})$ is defined by
$$\widetilde{L}^{\theta}_{T}(B^{s}_{p,r}):
=\{f\in
L^{\theta}(0,T;\mathcal{S}'):\|f\|_{\widetilde{L}^{\theta}_{T}(B^{s}_{p,r})}<+\infty\},$$
where
$$\|f\|_{\widetilde{L}^{\theta}_{T}(B^{s}_{p,r})}:=\Big(\sum_{q\geq-1}(2^{qs}\|\Delta_{q}f\|_{L^{\theta}_{T}(L^{p})})^{r}\Big)^{\frac{1}{r}}$$
with the usual convention if $r=\infty$.
\end{defn}

Furthermore, we set
$$\widetilde{\mathcal{C}}_{T}(B^{s}_{p,r}):=\widetilde{L}^{\infty}_{T}(B^{s}_{p,r})\cap\mathcal{C}([0,T],B^{s}_{p,r})
$$ and $$\widetilde{\mathcal{C}}^1_{T}(B^{s}_{p,r}):=\{f\in\mathcal{C}^1([0,T],B^{s}_{p,r})|\partial_{t}f\in\widetilde{L}^{\infty}_{T}(B^{s}_{p,r})\},$$
where the index $T>0$ will be omitted when $T=+\infty$.

\section{The L-P pointwise energy estimates}\setcounter{equation}{0}\label{sec:3}
In this section, we develop the L-P pointwise energy estimates and deduce the decay property for the linearized system (\ref{R-E4}).
Set
$$\tilde{w}:=(\sigma_{e},u_{e},\sigma_{i},u_{i},E).$$

\begin{prop}\label{prop3.1}
If $\tilde{w}_{0}\in \dot{B}^{\varrho}_{2,1}(\mathbb{R}^{3})\cap \dot{B}^{-s}_{2,\infty}(\mathbb{R}^{3})$ for $\varrho\geq0$ and $s>0$, then the solutions $\tilde{w}(t,x)$ of (\ref{R-E4}) has the decay estimate
\begin{equation}
\|\Lambda^{\ell}\tilde{w}\|_{B_{2,1}^{\varrho-\ell}}\lesssim \|\tilde{w}_{0}\|_{\dot{B}_{2,1}^{\varrho}\cap \dot{B}_{2,\infty}^{-s}}(1+t)^{-\frac{\ell+s}{2}} \label{R-E18}
\end{equation}
for $0\leq\ell\leq\varrho$. In particular, if $\tilde{w}_{0}\in \dot{B}^{\varrho}_{2,1}(\mathbb{R}^{3})\cap L^p(\mathbb{R}^{3})(1\leq p<2$), one further has
\begin{equation}
\|\Lambda^{\ell}\tilde{w}\|_{B_{2,1}^{\varrho-\ell}}\lesssim \|\tilde{w}_{0}\|_{\dot{B}_{2,1}^{\varrho}\cap L^{p}}(1+t)^{-\frac{3}{2}(\frac{1}{p}-\frac{1}{2})-\frac{\ell}{2}} \label{R-E19}
\end{equation}
for $0\leq\ell\leq\varrho$.
\end{prop}
\begin{proof}
Applying the inhomogeneous localization operator $\Delta_{q}(q\geq-1)$ to (\ref{R-E4}) gives
\begin{equation}
\left\{
\begin{array}{l}
\partial_{t} \Delta_{q}\sigma_{e}+\mathrm{div}\Delta_{q}u_{e}=0,\\
\partial_{t} \Delta_{q}u_{e}+\nabla\Delta_{q}\sigma_{e}+\Delta_{q}u_{e}-\Delta_{q}E=0,\\
\partial_{t} \Delta_{q}\sigma_{i}+\mathrm{div}\Delta_{q}u_{i}=0,\\
\partial_{t} \Delta_{q}u_{i}+\nabla\Delta_{q}\sigma_{i}+\Delta_{q}u_{i}+\Delta_{q}E=0,\\
\mathrm{div} \Delta_{q}E=\Delta_{q}\sigma_{e}-\Delta_{q}\sigma_{i}.
 \end{array}
\right.\label{R-E20}
\end{equation}

Next, by performing the Fourier transform  and then taking the inner product with $(\widehat{\Delta_{q}\sigma_{e}},\widehat{\Delta_{q}u_{e}},\widehat{\Delta_{q}\sigma_{i}},\widehat{\Delta_{q}u_{i}})$ respectively, we arrive at
\begin{eqnarray}
\frac{1}{2}\frac{d}{dt}|\widehat{\Delta_{q}w}|^2+(|\widehat{\Delta_{q}u_{e}}|^2+|\widehat{\Delta_{q}u_{i}}|^2)-\langle\widehat{\Delta_{q}E}, \widehat{\Delta_{q}u_{e}}-\widehat{\Delta_{q}u_{i}}\rangle=0, \label{R-E21}
\end{eqnarray}
where $w:=(\sigma_{e},u_{e},\sigma_{i},u_{i})$ and $\langle\cdot,\cdot\rangle$ denotes the inner product in $\mathbb{C}^{3}$.

For the term related the electron field $E$ of (\ref{R-E21}),  we have
\begin{eqnarray}
&&-\langle\widehat{\Delta_{q}E}, \widehat{\Delta_{q}u_{e}}-\widehat{\Delta_{q}u_{i}}\rangle\nonumber\\
&=& -\langle\widehat{\Delta_{q}\nabla\mathit\Phi}, \widehat{\Delta_{q}u_{e}}-\widehat{\Delta_{q}u_{i}}\rangle
\nonumber\\
&=& \langle\widehat{\Delta_{q}\mathit\Phi}, \widehat{\Delta_{q}\mathrm{div}u_{e}}-\widehat{\Delta_{q}\mathrm{div}u_{i}}\rangle
\nonumber\\
&=&  -\langle\widehat{\Delta_{q}\mathit\Phi},\widehat{\Delta_{q}\mathrm{div}E_{t}}\rangle
\nonumber\\
&=& \frac{1}{2}\frac{d}{dt} |\widehat{\Delta_{q}E}|^2. \label{R-E22}
\end{eqnarray}
Hence, it follows from (\ref{R-E21})-(\ref{R-E22}) that
\begin{eqnarray}
\frac{1}{2}\frac{d}{dt}|\widehat{\Delta_{q}\tilde{w}}|^2+(|\widehat{\Delta_{q}u_{e}}|^2+|\widehat{\Delta_{q}u_{i}}|^2)=0. \label{R-E23}
\end{eqnarray}

In order to create the desired dissipative inequality, we need to
rewrite (\ref{R-E4}) into the matrix form. Precisely,
\begin{equation}\partial_{t}w+\sum_{j=1}^{3}A_{j}(0)\partial_{x_{j}}w+Lw=G,\label{R-E24}
\end{equation}
with the coupled dynamic
field equation
\begin{equation}
\mathrm{div} E=\sigma_{e}-\sigma_{i}, \label{R-E25}
\end{equation}
where
$$ A_{j}(0):=\left(
                     \begin{array}{cccc}
                       0 & e_{j}^\top & 0 & 0 \\
                       e_{j} & 0 & 0 & 0 \\
                       0 & 0 & 0 & e_{j}^\top \\
                       0 & 0 & e_{j} & 0 \\
                     \end{array}
                   \right),\  L:=\left(
                     \begin{array}{cccc}
                       0 & 0 & 0 & 0 \\
                       0 & I_{3} & 0 & 0 \\
                       0 & 0 & 0 & 0 \\
                       0 & 0 & 0 & I_{3} \\
                     \end{array}
                   \right),\ G:=\left(
                                          \begin{array}{c}
                                            0 \\
                                            E \\
                                            0 \\
                                            -E \\
                                          \end{array}
                                        \right).$$
Note that $I_{3}$ is the unit matrix and $e_{j}$ is
$3$-dimensional vector where the $j$th component is one, others are
zero.

Now, we formulate a stability lemma, which has been well established
by the second author in \cite{SK} for generally hyperbolic-parabolic composite systems, and
sometimes referred to as the ``Shizuta-Kawashima condition".
\begin{lem}[Shizuta-Kawashima] \label{lem3.1}For all ~$\xi\in \mathbb{R}^{3},\
\xi\neq0$, there exists a real skew-symmetric smooth matrix $K(\xi)$
which is defined in the unit sphere $\mathbb{S}^{2}$:
\begin{eqnarray*}
K(\xi)=\left(%
\begin{array}{cccc}
 0 & \frac{\xi^\top}{|\xi|} & 0 & 0\\
  -\frac{\xi}{|\xi|} & 0 & 0 &0\\
0&0&0&\frac{\xi^\top}{|\xi|}\\
0&0&-\frac{\xi}{|\xi|}&0\\
\end{array}%
\right),
\end{eqnarray*}
such that
\begin{eqnarray}
K(\xi)\sum_{j=1}^{3}\xi_{j}A_{j}(0)=\left(%
\begin{array}{cccc}
  |\xi| & 0 & 0 & 0\\
  0 & -\frac{\xi\otimes\xi}{|\xi|} & 0 & 0\\
  0 & 0 & |\xi|&0\\
  0 & 0 & 0 & -\frac{\xi\otimes\xi}{|\xi|}\\
\end{array}%
\right),\label{R-E26}
\end{eqnarray}
where $A_{j}$ is the matrix appearing in (\ref{R-E24}).
\end{lem}

Applying the operator $\Delta_{q}\;(q\geq-1)$
to (\ref{R-E24}) gives
\begin{eqnarray}
\partial_{t}\Delta_{q}w+\sum_{j=1}^{3}A_{j}(0)\partial_{x_{j}}\Delta_{q}w+L\Delta_{q}w
=\Delta_{q}G.\label{R-E27}
\end{eqnarray}

Perform the Fourier transform with respect to the space variable $x$
for (\ref{R-E24}) before multiplying the matrix $-i|\xi|K(\xi)$. By taking
the inner product in the resulting equality with $\widehat{\Delta_{q}w}$, and then choosing the real part of each term, we get
\begin{eqnarray}
&&\frac{1}{2}\frac{d}{dt}\mathrm{Im}\langle|\xi|K(\xi)\widehat{\Delta_{q}w},\widehat{\Delta_{q}w}\rangle
+|\xi|\langle K(\xi)\sum_{j=1}^{3}\xi_{j}A_{j}(0)\widehat{\Delta_{q}w},\widehat{\Delta_{q}w}\rangle
\nonumber\\&=&|\xi|\mathrm{Im}\langle\tilde{K}(\xi)L\widehat{\Delta_{q}w},\widehat{\Delta_{q}w}\rangle+
|\xi|\mathrm{Im}\langle\tilde{K}(\xi)\widehat{\Delta_{q}G},\widehat{\Delta_{q}w}\rangle. \label{R-E28}
\end{eqnarray}

According to Lemma \ref{lem3.1}, the second term of the left-hand of (\ref{R-E28}) is bounded from below by
\begin{eqnarray}
&&|\xi|\langle K(\xi)\sum_{j=1}^{3}\xi_{j}A_{j}(0)\widehat{\Delta_{q}w},\widehat{\Delta_{q}w}\rangle\nonumber\\&\geq & |\xi|^2|\widehat{\Delta_{q}w}|^2-2|\xi|^2(|\widehat{\Delta_{q}u_{e}}|^2+|\widehat{\Delta_{q}u_{i}}|^2). \label{R-E29}
\end{eqnarray}
Moreover,  by virtue of Young's inequality, the first term of the right side of (\ref{R-E28}) can be estimated as
\begin{eqnarray}
|\xi|\Big|\mathrm{Im}\langle\tilde{K}(\xi)L\widehat{\Delta_{q}w},\widehat{\Delta_{q}w}\rangle\Big|\leq \epsilon |\xi|^2|\widehat{\Delta_{q}w}|^2
+C_{\epsilon}(|\widehat{\Delta_{q}u_{e}}|^2+|\widehat{\Delta_{q}u_{i}}|^2), \label{R-E30}
\end{eqnarray}
where $\epsilon$ is a small constant to be determined and $C_{\epsilon}:=C(\epsilon)$.

For the second term of the right side of (\ref{R-E28}), we have
\begin{eqnarray}
&&|\xi|\mathrm{Im}\langle\tilde{K}(\xi)\widehat{\Delta_{q}G},\widehat{\Delta_{q}w}\rangle
\nonumber\\&=&
\mathrm{Im}\Big(\overline{\widehat{\Delta_{q}\sigma_{e}}-\widehat{\Delta_{q}\sigma_{i}}}\xi^{\top}\widehat{\Delta_{q}E}\Big)
\nonumber\\&=&-\frac{1}{2}i\overline{\widehat{\Delta_{q}\sigma_{e}}-\widehat{\Delta_{q}\sigma_{i}}}\xi^{\top}\widehat{\Delta_{q}E}
+\frac{1}{2}i\Big(\widehat{\Delta_{q}\sigma_{e}}-\widehat{\Delta_{q}\sigma_{i}}\Big)\xi^{\top}\overline{\widehat{\Delta_{q}E}}
\nonumber\\&=&-\frac{1}{2}\overline{\widehat{\Delta_{q}(\sigma_{e}-\sigma_{i})}}\widehat{\Delta_{q}\mathrm{div}E}
-\frac{1}{2}\widehat{\Delta_{q}(\sigma_{e}-\sigma_{i})}\overline{\widehat{\Delta_{q}\mathrm{div}E}}
\nonumber\\&=& -|\widehat{\Delta_{q}\mathrm{div}E}|^2, \label{R-E31}
\end{eqnarray}
where we have used the equation $\mathrm{div}E=\sigma_{e}-\sigma_{i}$.

Together with (\ref{R-E28})-(\ref{R-E31}), we conclude that
\begin{eqnarray}
&&\frac{1}{2}\frac{d}{dt}\mathrm{Im}\Big\langle\frac{|\xi|}{1+|\xi|^2}\tilde{K}(\xi)\widehat{\Delta_{q}w},\widehat{\Delta_{q}w}\Big\rangle
+\frac{|\xi|^2}{2(1+|\xi|^2)}|\widehat{\Delta_{q}w}|^2\nonumber\\ &&\hspace{5mm}+\frac{1}{1+|\xi|^2}|\widehat{\Delta_{q}\mathrm{div}E}|^2
\nonumber\\ &\leq& C(|\widehat{\Delta_{q}u_{e}}|^2+|\widehat{\Delta_{q}u_{i}}|^2), \label{R-E32}
\end{eqnarray}
where we have taken $\epsilon=1/2.$
Therefore, it follows from (\ref{R-E32}) and (\ref{R-E23}) that (\ref{R-E4})
admits a frequency-localization Lyapunov function of the form
$$\mathcal{E}[\widehat{\Delta_{q}\tilde{w}}]:=|\widehat{\Delta_{q}\tilde{w}}|^2
+\kappa\mathrm{Im}\Big\langle\frac{|\xi|}{1+|\xi|^2}K(\xi)\widehat{\Delta_{q}w},\widehat{\Delta_{q}w}\Big\rangle$$
such that
\begin{eqnarray}
&&\frac{1}{2}\frac{d}{dt}\mathcal{E}[\widehat{\Delta_{q}\tilde{w}}]+(1-\kappa)(|\widehat{\Delta_{q}u_{e}}|^2+|\widehat{\Delta_{q}u_{i}}|^2)\nonumber\\ &&
+\frac{\alpha|\xi|^2}{2(1+|\xi|^2)}|\widehat{\Delta_{q}w}|^2+\frac{\alpha}{1+|\xi|^2}|\widehat{\Delta_{q}\mathrm{div}E}|^2\leq0, \label{R-E33}
\end{eqnarray}
where $\kappa>0$ is some small constant.

Choosing $\kappa$ sufficiently small such that $1-\kappa>0$ and $E[\widehat{\Delta_{q}\tilde{w}}]\approx|\widehat{\Delta_{q}\tilde{w}}|^2.$
Furthermore, we deduce that
\begin{eqnarray}
\frac{1}{2}\frac{d}{dt}\mathcal{E}[\widehat{\Delta_{q}\tilde{w}}]+\frac{\kappa|\xi|^2}{2(1+|\xi|^2)}|\widehat{\Delta_{q}w}|^2
+\frac{\kappa}{1+|\xi|^2}|\widehat{\Delta_{q}\mathrm{div}E}|^2\leq0. \label{R-E34}
\end{eqnarray}

In the following,  we deal with (\ref{R-E34}) at the high-frequency and low-frequency, respectively.\\

\underline{\textit{Case 1}($q\geq0$)}
In this case, since $|\xi|\sim 2^{q}\geq1$, we arrive at
\begin{eqnarray}
\frac{1}{2}\frac{d}{dt}\int_{\mathbb{R}^3_{\xi}}\mathcal{E}[\widehat{\Delta_{q}\tilde{w}}]+\frac{\kappa}{4}\|\widehat{\Delta_{q}w}\|^2_{L^2}
+\kappa\int_{\mathbb{R}^3_{\xi}}\frac{1}{1+|\xi|^2}|\widehat{\Delta_{q}\mathrm{div}E}|^2\leq0. \label{R-E35}
\end{eqnarray}
With the aid of Plancherel's theorem and the irrotationality of $E$, the second term on the left of (\ref{R-E35}) can be estimated below
\begin{eqnarray}
&&\int_{\mathbb{R}^3_{\xi}}\frac{1}{1+|\xi|^2}|\widehat{\Delta_{q}\mathrm{div}E}|^2\nonumber\\&\approx&
\|(1-\Delta)^{-1/2}\Delta_{q}\mathrm{div}E\|^2_{L^2}
\nonumber\\&\approx&
\|(1-\Delta)^{-1/2}\nabla\Delta_{q}E\|^2_{L^2}\nonumber\\&\approx&
\int_{\mathbb{R}^3_{\xi}}\frac{|\xi|^2}{1+|\xi|^2}|\widehat{\Delta_{q}E}|^2\geq\frac{1}{2}\|\widehat{\Delta_{q}E}\|^2_{L^2}. \label{R-E36}
\end{eqnarray}
Therefore, combing (\ref{R-E35})-(\ref{R-E36}), there exists a constant $c_{1}>0$ such that
\begin{eqnarray}
\|\Delta_{q}\tilde{w}\|_{L^2}\lesssim e^{-c_{1}t}\|\Delta_{q}\tilde{w}_{0}\|_{L^2}, \label{R-E37}
\end{eqnarray}
which implies that
\begin{eqnarray}
&&\sum_{q\geq0}2^{q(\varrho-\ell)}\|\Delta_{q}\Lambda^{\ell}\tilde{w}\|_{L^2}\nonumber\\&\lesssim& e^{-c_{1}t}\sum_{q\geq0}2^{q(\varrho-\ell)}\|\Delta_{q}\Lambda^{\ell}\tilde{w}_{0}\|_{L^2}\lesssim e^{-c_{1}t} \|\tilde{w}_{0}\|_{\dot{B}_{2,1}^{\varrho}}. \label{R-E38}
\end{eqnarray}

\underline{\textit{Case 2}($q=-1$)}

In this case, since $|\xi|\leq1$, we get
\begin{eqnarray}
\frac{1}{2}\frac{d}{dt}\mathcal{E}[\widehat{\Delta_{-1}\tilde{w}}]+\frac{\kappa|\xi|^2}{4}|\widehat{\Delta_{-1}w}|^2
+\frac{\kappa}{2}|\widehat{\Delta_{-1}\mathrm{div}E}|^2\leq0. \label{R-E39}
\end{eqnarray}
Multiplying (\ref{R-E39}) with $|\xi|^{2\ell}$ and integrating the resulting inequality over $\mathbb{R}^{3}_{\xi}$, similar to the computation of (\ref{R-E36}), we conclude that there exists a constant $c_{2}>0$ such that
\begin{eqnarray}
\frac{1}{2}\frac{d}{dt}\int_{\mathbb{R}^3_{\xi}}|\xi|^{2\ell}\mathcal{E}[\widehat{\tilde{w}_{-1}}]+c_{2}\|\Lambda^{\ell+1}\tilde{w}_{-1}\|^2_{L^2}\leq0, \label{R-E40}
\end{eqnarray}
where $$\Delta_{-1}\tilde{w}:=\tilde{w}_{-1},\ \ \
\int_{\mathbb{R}^3_{\xi}}|\xi|^{2\ell}\mathcal{E}[\widehat{\tilde{w}_{-1}}]\approx\|\Lambda^{\ell}\tilde{w}_{-1}\|^2_{L^2}.$$
According to the interpolation inequality related the Besov space $\dot{B}^{-s}_{2,\infty}$ (see Lemma \ref{lem8.2}), we have
\begin{eqnarray}
\|\Lambda^{\ell}\tilde{w}_{-1}\|_{L^2} &\lesssim&\|\Lambda^{\ell+1}\tilde{w}_{-1}\|^{\theta}_{L^2}\|\tilde{w}_{-1}\|^{1-\theta}_{\dot{B}^{-s}_{2,\infty}}\ \ \  \Big(\theta=\frac{\ell+s}{\ell+1+s}\Big)\nonumber\\ &\lesssim& \|\Lambda^{\ell+1}\tilde{w}_{-1}\|^{\theta}_{L^2}\|\tilde{w}\|^{1-\theta}_{\dot{B}^{-s}_{2,\infty}}, \label{R-E41}
\end{eqnarray}

On the other hand, by employing the operator $\dot{\Delta}_{q}(q\in\mathbb{Z})$ to (\ref{R-E4}) and performing
the procedure leading to (\ref{R-E23}), we can obtain
\begin{eqnarray}
\|\tilde{w}\|_{\dot{B}^{-s}_{2,\infty}}\leq\|\tilde{w}_{0}\|_{\dot{B}^{-s}_{2,\infty}}.  \label{R-E42}
\end{eqnarray}
Hence, together with (\ref{R-E40})-(\ref{R-E42}), we are led to the differential inequality
\begin{eqnarray}
\frac{d}{dt}\int_{\mathbb{R}^3_{\xi}}|\xi|^{2\ell}\mathcal{E}[\widehat{\tilde{w}_{-1}}]
+C\|\tilde{w}_{0}\|_{\dot{B}^{-s}_{2,\infty}}^{-\frac{2}{s}}(\|\Lambda^{\ell}\tilde{w}_{-1}\|^2_{L^2})^{1+\frac{1}{\ell+s}}\leq0,\label{R-E43}
\end{eqnarray}
which implies that
\begin{eqnarray}
\|\Lambda^{\ell}\tilde{w}_{-1}\|_{L^2}\lesssim\|\tilde{w}_{0}\|_{\dot{B}^{-s}_{2,\infty}}(1+t)^{-\frac{\ell+s}{2}}.\label{R-E44}
\end{eqnarray}
Hence, it follows from the high-frequency estimate (\ref{R-E38}) and low-frequency estimate (\ref{R-E44}) that
\begin{eqnarray}
&&\|\Lambda^{\ell}\tilde{w}\|_{B_{2,1}^{\varrho-\ell}}\nonumber\\&\lesssim& \|\tilde{w}_{0}\|_{\dot{B}^{-s}_{2,\infty}}(1+t)^{-\frac{\ell+s}{2}}+ \|\tilde{w}_{0}\|_{\dot{B}_{2,1}^{\varrho}}e^{-c_{1}t}\nonumber\\&\lesssim& \|\tilde{w}_{0}\|_{\dot{B}_{2,1}^{\varrho}\cap \dot{B}_{2,\infty}^{-s}}(1+t)^{-\frac{\ell+s}{2}}. \label{R-E45}
\end{eqnarray}

Finally, the optimal decay estimate (\ref{R-E19}) direcely follows from the embedding $L^p(\mathbb{R}^{3})\hookrightarrow\dot{B}^{-s}_{2,\infty}(\mathbb{R}^{3})(s=3(1/p-1/2))$ in Lemma \ref{lem8.5}. Therefore, the proof of Proposition \ref{prop3.1} is complete.
\end{proof}

Additionally, we have also the decay property on the framework of homogeneous Besov spaces.
\begin{prop}\label{prop3.2}
If $\tilde{w}_{0}\in \dot{B}^{\varrho}_{2,1}(\mathbb{R}^{3})\cap \dot{B}^{-s}_{2,\infty}(\mathbb{R}^{3})$ for $\varrho\in \mathbb{R}, s\in \mathbb{R}$ satisfying $\varrho+s>0$, then the solution $\tilde{w}(t,x)$ of (\ref{R-E4}) has the decay estimate
\begin{equation}
\|\tilde{w}\|_{\dot{B}_{2,1}^{\varrho}}\lesssim \|\tilde{w}_{0}\|_{\dot{B}_{2,1}^{\varrho}\cap \dot{B}_{2,\infty}^{-s}}(1+t)^{-\frac{\varrho+s}{2}}. \label{R-E46}
\end{equation}
In particular, if $\tilde{w}_{0}\in \dot{B}^{\varrho}_{2,1}(\mathbb{R}^{3})\cap L^p(\mathbb{R}^{3})(1\leq p<2$), one further has
\begin{equation}
\|\tilde{w}\|_{\dot{B}_{2,1}^{\varrho}}\lesssim \|\tilde{w}_{0}\|_{\dot{B}_{2,1}^{\varrho}\cap L^{p}}(1+t)^{-\frac{3}{2}(\frac{1}{p}-\frac{1}{2})-\frac{\varrho}{2}}. \label{R-E47}
\end{equation}
\end{prop}
\begin{proof}
It suffices to show the different low-frequency estimate, since the operator $\dot{\Delta}_{q}$ consists with $\Delta_{q}$ for $q\geq0$. Note that the
irrotationality of $E$, the proof can be finished by the similar procedure as in \cite{XK2}. We feel free to skip the details for brevity.
\end{proof}

\section{Localized time-weighted energy approaches}\setcounter{equation}{0}\label{sec:4}
The aim of this section is to deduce decay estimates for
the nonlinear system (\ref{R-E3})-(\ref{R-E333}). For this purpose, the frequency-localization
Duhamel principle and time-weighted energy approaches in terms of the low-frequency and high-frequency decomposition
are mainly developed.

System (\ref{R-E3}) can be written as the following form for $\tilde{w}=(\sigma_{e},u_{e},\sigma_{i},u_{i},E)$:
\begin{equation}
\left\{
\begin{array}{l}
\partial_{t} \sigma_{e}+\mathrm{div}u_{e}=f_{1e},\\
\partial_{t} u_{e}+\nabla\sigma_{e}+u_{e}-E=f_{2e},\\
\partial_{t} \sigma_{i}+\mathrm{div}u_{i}=f_{1i},\\
\partial_{t} u_{i}+\nabla\sigma_{i}+u_{i}+E=f_{2i},\\
\partial_{t} E=-\nabla\Delta^{-1}\mathrm{div}(u_{e}-u_{i})+f_{3},
 \end{array}
\right.\label{R-E48}
\end{equation}
with
$$f_{1a}:=-u_{a}\cdot\nabla\sigma_{a}-\sigma_{a}\mathrm{div}u_{a}=-\mathrm{div}(\sigma_{a}u_{a}), \ \ \ f_{2a}:=-u_{a}\cdot\nabla u_{a}-h(\sigma_{a})\nabla\sigma_{a},$$
$$f_{3}:=-\nabla\Delta^{-1}\mathrm{div}(\sigma_{e}u_{e}-\sigma_{i}u_{i}),
$$
where we note that the electronic field equation in (\ref{R-E3}) can be replaced equivalently by the nonlocal evolutionary equation for $E$.
The nonlocal term $-\nabla\Delta^{-1}\mathrm{div}f$ means the sum of products of Riesz transforms of $f$.

The initial data (\ref{R-E333}) is given correspondingly by
\begin{eqnarray}
\tilde{w}|_{t=0}=\tilde{w}_{0}(x)=(\sigma_{e0},u_{e0},\sigma_{i0},u_{i0},E_{0}) \label{R-E488}
\end{eqnarray}
with $E_{0}:=\nabla\Delta^{-1}(\sigma_{e0}-\sigma_{i0}).$

Firstly, we denote by $\mathcal{G}(t)$ the Green matrix associated with the linearized system (\ref{R-E48})-(\ref{R-E488}):
$$\mathcal{G}(t)f=\mathcal{F}^{-1}[e^{-\hat{A}(\xi)t}\mathcal{F}f],$$
with
\begin{eqnarray*}
\hat{A}(\xi)=\left(%
\begin{array}{ccccc}
 0 & i\xi^{\top} & 0 & 0& 0\\
 i\xi & I_{3} & 0 &0 &-I_{3}\\
0 & 0 &  i\xi & I_{3}& I_{3}\\
0&\frac{\xi\otimes\xi}{|\xi|^2}&0&-\frac{\xi\otimes\xi}{|\xi|^2}& 0\\
\end{array}%
\right).
\end{eqnarray*}
Then the solution of
(\ref{R-E4}) with the initial data $\tilde{w}_{0}$ is given by $\mathcal{G}(t)\tilde{w}_{0}$. Furthermore,
by the standard Duhamel principle,  the solution of (\ref{R-E48})-(\ref{R-E488}) can be expressed as
\begin{eqnarray}
\tilde{w}(t,x)=\mathcal{G}(t)\tilde{w}_{0}+\int^{t}_{0}\mathcal{G}(t-\tau)\mathcal{R}(\tau)d\tau, \label{R-E49}
\end{eqnarray}
where $\mathcal{R}:=(f_{1e},f_{2e},f_{1i},f_{2i},f_{3})^{\top}$. It is not difficult to prove
the frequency-localization Duhamel principle for (\ref{R-E48})-(\ref{R-E488}).

\begin{lem}\label{lem4.1}
Suppose that $\tilde{w}(t,x)$ is a solution of (\ref{R-E48})-(\ref{R-E488}). Then
\begin{eqnarray}
\Delta_{q}\Lambda^{\ell}\tilde{w}(t,x)=\Delta_{q}\Lambda^{\ell}[\mathcal{G}(t)\tilde{w}_{0}]
+\int^{t}_{0}\Delta_{q}\Lambda^{\ell}[\mathcal{G}(t-\tau)\mathcal{R}(\tau)]d\tau\label{R-E50}
\end{eqnarray}
for $q\geq-1$ and $\ell\in \mathbb{R}$, and
\begin{eqnarray}
\dot{\Delta}_{q}\Lambda^{\ell}\tilde{w}(t)=\dot{\Delta}_{q}\Lambda^{\ell}[\mathcal{G}(t)\tilde{w}_{0}]
+\int^{t}_{0}\dot{\Delta}_{q}\Lambda^{\ell}[\mathcal{G}(t-\tau)\mathcal{R}(\tau)]d\tau\label{R-E51}
\end{eqnarray}
for $q\in\mathbb{Z}$ and $\ell\in \mathbb{R}$.
\end{lem}

In what follows, the main task is to prove the decay estimates by using the time-weighted energy approach which was initialled in \cite{Ma}. To do this, we first define some time-weighted sup-norms as follows:
$$
\mathcal{E}_{0}(t):=\sup_{0\leq\tau\leq t}\|\tilde{w}(\tau)\|_{B^{s_{c}}_{2,1}};
$$
\begin{eqnarray*}
\mathcal{E}_{1}(t)&:=&\sup_{0\leq\ell<(s_{c}-1)}\sup_{0\leq\tau\leq t}(1+\tau)^{\frac{s+\ell}{2}}\|\Lambda^{\ell}\tilde{w}(\tau)\|_{B^{s_{c}-1-\ell}_{2,1}}\nonumber\\&&+\sup_{0\leq\tau\leq t}(1+\tau)^{\frac{s+s_{c}-1}{2}}\|\Lambda^{s_{c}-1}\tilde{w}(\tau)\|_{\dot{B}^{0}_{2,1}};
\end{eqnarray*}
\begin{eqnarray*}
\mathcal{E}_{2}(t)&:=&\sup_{0\leq\ell<(s_{c}-2)}\sup_{0\leq\tau\leq t}(1+\tau)^{\frac{s+\ell+1}{2}}\|\Lambda^{\ell}(u_{e},u_{i})(\tau)\|_{B^{s_{c}-2-\ell}_{2,1}}
\nonumber\\&&+\sup_{0\leq\tau\leq t}(1+\tau)^{\frac{s+s_{c}-1}{2}}\|\Lambda^{\sigma_{c}-2}(u_{e},u_{i})(\tau)\|_{\dot{B}^{0}_{2,1}}
\end{eqnarray*}
and further set
$$\mathcal{E}(t):=\mathcal{E}_{1}(t)+\mathcal{E}_{2}(t).$$

\begin{rem} \label{rem4.1}
In comparison with \cite{Ma}, the new energy functionals contain different time-weighted norms according to the derivative index, since we take care of the topological relation between inhomegeous Besov spaces and homogeneous Besov spaces to overcome the technical difficulty in the subsequent nonlinear analysis. In addition, improved Bernstein inequality (Lemma \ref{lem2.1}) allows energy functionals to have the derivative case of fractional order rather than the integer order only.
\end{rem}

Precisely, with aid of the frequency-localization Duhamel principle in Lemma \ref{lem4.1},
we shall develop the time-weighted energy approach in terms of the low-frequency and high-frequency decomposition.
Consequently, we prove the following result.
\begin{prop}\label{prop4.1}
Let $\tilde{w}=(\sigma_{e},u_{e},\sigma_{i},u_{i},E)$ be the global classical solution in the sense of Theorem \ref{thm1.1}. Suppose that $\tilde{w}_{0}\in B^{s_{c}}_{2,1}\cap \dot{B}^{-s}_{2,\infty}(0<s\leq 3/2)$ and the norm
$\mathcal{M}_{0}:=\|\tilde{w}_{0}\|_{B^{s_{c}}_{2,1}\cap \dot{B}^{-s}_{2,\infty}}$ is sufficiently small. Then it holds that
\begin{eqnarray}
\|\Lambda^{\ell}\tilde{w}(t)\|_{X_{1}}\lesssim \mathcal{M}_{0}(1+t)^{-\frac{s+\ell}{2}} \label{R-E52}
\end{eqnarray}
for $0\leq\ell\leq s_{c}-1$, where $X_{1}:=B^{s_{c}-1-\ell}_{2,1}$ if $0\leq\ell<s_{c}-1$ and $X_{1}:=\dot{B}^{0}_{2,1}$ if $\ell=s_{c}-1$; and
\begin{eqnarray}
\|\Lambda^{\ell}(u_{e},u_{i},\sigma_{e}-\sigma_{i})(t)\|_{X_{2}}\lesssim \mathcal{M}_{0}(1+t)^{-\frac{s+\ell+1}{2}} \label{R-E53}
\end{eqnarray}
for $0\leq\ell\leq s_{c}-2$, where $X_{2}:=B^{s_{c}-2-\ell}_{2,1}$ if $0\leq\ell<s_{c}-2$ and $X_{2}:=\dot{B}^{0}_{2,1}$ if $\ell=s_{c}-2$.
\end{prop}

Proposition \ref{prop4.1} mainly depends on an energy inequality related to those time-weighted quantities, which is included in the following proposition.
\begin{prop}\label{prop4.2}
Let $\tilde{w}=(\sigma_{e},u_{e},\sigma_{i},u_{i},E)$ be the global classical solution in the sense of Theorem \ref{thm1.1}. Additional, if $\tilde{w}_{0}\in \dot{B}^{-s}_{2,\infty}(0<s\leq 3/2)$, then
\begin{eqnarray}
\mathcal{E}(t)\lesssim \mathcal{M}_{0}+\mathcal{E}^{2}(t)+\mathcal{E}_{0}(t)\mathcal{E}(t), \label{R-E54}
\end{eqnarray}
where $\mathcal{M}_{0}$ is defined as Proposition \ref{prop4.1}.
\end{prop}

The proof of Proposition \ref{prop4.2} is divided into several lemmas for clarity.
The first lemma is about the nonlinear low-frequency estimates of solutions.

\begin{lem}\label{lem4.2} (Low-frequency estimates)
Under the assumption of Proposition \ref{prop4.2}, we have
\begin{eqnarray}
\|\Delta_{-1}\Lambda^{\ell}\tilde{w}\|_{L^2}\lesssim \|\tilde{w}_{0}\|_{\dot{B}^{-s}_{2,\infty}} (1+t)^{-\frac{s+\ell}{2}}+(1+t)^{-\frac{s+\ell}{2}}\mathcal{E}^{2}(t). \label{R-E55}
\end{eqnarray}
for $0\leq\ell<\sigma_{c}-1$ and
\begin{eqnarray}
\sum_{q<0}\|\dot{\Delta}_{q}\Lambda^{\sigma_{c}-1}\tilde{w}\|_{L^2}\lesssim \|\tilde{w}_{0}\|_{\dot{B}^{-s}_{2,\infty}} (1+t)^{-\frac{s+\sigma_{c}-1}{2}}+(1+t)^{-\frac{s+\sigma_{c}-1}{2}}\mathcal{E}^{2}(t). \label{R-E56}
\end{eqnarray}
\end{lem}

\begin{proof} From (\ref{R-E44}), we have
\begin{eqnarray}
\|\Delta_{-1}\Lambda^{\ell}[\mathcal{G}(t)\tilde{w}_{0}]\|_{L^2}\lesssim\|\tilde{w}_{0}\|_{\dot{B}^{-s}_{2,\infty}}(1+t)^{-\frac{s+\ell}{2}}.\label{R-E57}
\end{eqnarray}
Furthermore, it follows from Lemma \ref{lem4.1} that
\begin{eqnarray}
&&\|\Delta_{-1}\Lambda^{\ell}\tilde{w}(t,x)\|_{L^2}\nonumber\\&\leq&\|\Delta_{-1}\Lambda^{\ell}[\mathcal{G}(t)\tilde{w}_{0}]\|_{L^2}
+\int^{t}_{0}\|\Delta_{-1}\Lambda^{\ell}[\mathcal{G}(t-\tau)\mathcal{R}(\tau)]\|_{L^2}d\tau
\nonumber\\&\lesssim&\|\tilde{w}_{0}\|_{\dot{B}^{-s}_{2,\infty}}(1+t)^{-\frac{s+\ell}{2}}
+\int_{0}^{t}(1+t-\tau)^{-\frac{s+\ell}{2}}\|\mathcal{R}(\tau)\|_{\dot{B}^{-s}_{2,\infty}}d\tau. \label{R-E58}
\end{eqnarray}
Next, we turn to estimate the norm $\|\mathcal{R}(\tau)\|_{\dot{B}^{-s}_{2,\infty}}$. For instance, we obtain
\begin{eqnarray}
\|f_{3}(\tau)\|_{\dot{B}^{-s}_{2,\infty}}\leq \|(\sigma_{e}u_{e}-\sigma_{i}u_{i})(\tau)\|_{\dot{B}^{-s}_{2,\infty}}\leq \|\sigma_{e}u_{e}(\tau)\|_{L^p}+\|\sigma_{i}u_{i}(\tau)\|_{L^p}, \label{R-E59}
\end{eqnarray}
where we have used the $L^2$-boundedness of Riesz transform on each block $\dot{\Delta}_{q}(q\in \mathbb{Z})$ and the embedding $L^p(\mathbb{R}^{3})\hookrightarrow\dot{B}^{-s}_{2,\infty}(\mathbb{R}^{3})(s=3(1/p-1/2))$ in Lemma \ref{lem8.5}.

In the following, we proceed with the inequality (\ref{R-E59}) with aid of different interpolation inequalities in Lemma \ref{lem8.4}.\\

\underline{\textit{Case 1} ($0<s\leq 1/2$)}
It follows from the H\"{o}lder's inequality that
\begin{eqnarray}
\|(\sigma_{e}u_{e})(\tau)\|_{L^p}\leq\|\sigma_{e}(\tau)\|_{L^{3/s}}\|u_{e}(\tau)\|_{L^2}, \label{R-E60}
\end{eqnarray}
since $1/p=s/3+1/2.$

Applying Lemma \ref{lem8.4} (taking $r=2$) and Young's inequality to (\ref{R-E60}) gives
\begin{eqnarray}
\|(\sigma_{e}u_{e})(\tau)\|_{L^p}&\lesssim & \|\Lambda\sigma_{e}(\tau)\|^{\theta}_{L^{2}}\|\Lambda^{\beta}\sigma_{e}(\tau)\|^{1-\theta}_{L^{2}}\|u_{e}(\tau)\|_{L^2}
\nonumber\\&
\lesssim & (\|\Lambda\sigma_{e}(\tau)\|_{L^{2}}+\|\Lambda^{\beta}\sigma_{e}(\tau)\|_{L^{2}})\|u_{e}(\tau)\|_{L^2}, \label{R-E61}
\end{eqnarray}
where $3/2-s<\beta\leq s_{c}-1$ and $\theta=\frac{\beta+s-3/2}{\beta-1}$. In case that $3/2-s<\beta<s_{c}-1$, we arrive at
\begin{eqnarray}
&&\|(\sigma_{e}u_{e})(\tau)\|_{L^p}\nonumber\\&\lesssim &
(\|\Lambda\sigma_{e}(\tau)\|_{B^{s_{c}-2}_{2,1}}+\|\Lambda^{\beta}\sigma_{e}(\tau)\|_{B^{s_{c}-1-\beta}_{2,1}})\|u_{e}(\tau)\|_{B^{s_{c}-2}_{2,1}}
\nonumber\\&\lesssim &
\Big[(1+\tau)^{-\frac{s}{2}-\frac{1}{2}}+(1+\tau)^{-\frac{s}{2}-\frac{\beta}{2}}\Big](1+\tau)^{-\frac{s}{2}-\frac{1}{2}}\mathcal{E}_{1}(t)\mathcal{E}_{2}(t)
\nonumber\\&
\lesssim &(1+\tau)^{-s-1}\mathcal{E}_{1}(t)\mathcal{E}_{2}(t). \label{R-E62}
\end{eqnarray}
In case that $\beta=s_{c}-1$, we have
\begin{eqnarray}
&&\|(\sigma_{e}u_{e})(\tau)\|_{L^p}\nonumber\\&\lesssim &
(\|\Lambda\sigma_{e}(\tau)\|_{B^{s_{c}-2}_{2,1}}+\|\Lambda^{s_{c}-1}\sigma_{e}(\tau)\|_{\dot{B}^{0}_{2,1}})\|u_{e}(\tau)\|_{B^{s_{c}-2}_{2,1}}
\nonumber\\&
\lesssim &(1+\tau)^{-s-1}\mathcal{E}_{1}(t)\mathcal{E}_{2}(t), \label{R-E63}
\end{eqnarray}
where we used the fact $\dot{B}^{0}_{2,1}\hookrightarrow L^2$.

\underline{\textit{Case 2} ($1/2<s\leq 3/2$)}

It follows from the H\"{o}lder's inequality,
Lemma \ref{lem8.4} and Young's inequality that
\begin{eqnarray}
&&\|(\sigma_{e}u_{e})(\tau)\|_{L^p}\nonumber\\ &\lesssim & \|\sigma_{e}(\tau)\|_{L^{n/s}}\|u_{e}(\tau)\|_{L^2}
\nonumber\\&
\lesssim &\|\sigma_{e}(\tau)\|^{\theta}_{L^{2}}\|\Lambda \sigma_{e}(\tau)\|^{1-\theta}_{L^2}\|u_{e}(\tau)\|_{L^2}
\nonumber\\&
\lesssim &(\|\sigma_{e}(\tau)\|_{L^{2}}+\|\Lambda \sigma_{e}(\tau)\|_{L^2})\|u_{e}(\tau)\|_{L^2}
\nonumber\\&
\lesssim &(\|\sigma_{e}(\tau)\|_{B^{s_{c}-1}_{2,1}}+\|\Lambda\sigma_{e}(\tau)\|_{B^{s_{c}-2}_{2,1}})\|u_{e}\|_{B^{\sigma_{c}-2}_{2,1}}
\nonumber\\&
\lesssim &\Big[(1+\tau)^{-\frac{s}{2}}+(1+\tau)^{-\frac{s}{2}-\frac{1}{2}}\Big](1+\tau)^{-\frac{s}{2}-\frac{1}{2}}\mathcal{E}_{1}(t)\mathcal{E}_{2}(t)
\nonumber\\&
\lesssim& (1+\tau)^{-s-\frac{1}{2}}\mathcal{E}_{1}(t)\mathcal{E}_{2}(t) \label{R-E64}
\end{eqnarray}
with $\theta=1+s-3/2$, where $s+1/2>0$.

Similarly,
\begin{eqnarray}
 \|(\sigma_{i}u_{i})(\tau)\|_{L^p}\lesssim
 \begin{cases}
 (1+\tau)^{-s-1}\mathcal{E}_{1}(t)\mathcal{E}_{2}(t),\ \ \ \ 0<s\leq 1/2;\\
 (1+\tau)^{-s-\frac{1}{2}}\mathcal{E}_{1}(t)\mathcal{E}_{2}(t),\ \ \ 1/2<s\leq 3/2.
 \end{cases}\label{R-E65}
\end{eqnarray}
Hence, together with (\ref{R-E64})-(\ref{R-E65}), we are led to the estimate
\begin{eqnarray}
\|f_{3}(\tau)\|_{\dot{B}^{-s}_{2,\infty}}\lesssim
 \begin{cases}
 (1+\tau)^{-s-1}\mathcal{E}_{1}(t)\mathcal{E}_{2}(t),\ \ \ \ 0<s\leq 1/2;\\
 (1+\tau)^{-s-\frac{1}{2}}\mathcal{E}_{1}(t)\mathcal{E}_{2}(t),\ \ \ 1/2<s\leq 3/2.
 \end{cases} \label{R-E66}
\end{eqnarray}
Furthermore, in a similar way, we can arrive at
\begin{eqnarray}
\|f_{ja}(\tau)\|_{\dot{B}^{-s}_{2,\infty}}\lesssim
 \begin{cases}
 (1+\tau)^{-s-1}\mathcal{E}_{1}^2(t),\ \ \ \ 0<s\leq 1/2;\\
 (1+\tau)^{-s-\frac{1}{2}}\mathcal{E}_{1}^2(t),\ \ \ 1/2<s\leq 3/2,
 \end{cases}\label{R-E67}
\end{eqnarray}
where $j=1,2$ and $a=e,i$. Note that the definition of $\mathcal{E}(t)$, we conclude that
\begin{eqnarray}
\|\mathcal{R}(\tau)\|_{\dot{B}^{-s}_{2,\infty}}\lesssim
 \begin{cases}
 (1+\tau)^{-s-1}\mathcal{E}^2(t),\ \ \ \ 0<s\leq 1/2;\\
 (1+\tau)^{-s-\frac{1}{2}}\mathcal{E}^2(t),\ \ \ 1/2<s\leq 3/2,
 \end{cases}\label{R-E68}
\end{eqnarray}
Therefore, the desired inequality (\ref{R-E55}) is followed by (\ref{R-E58}) and (\ref{R-E68}) immediately.

On the other hand, as $\ell=s_{c}-1$, it follows from Proposition \ref{prop3.2} that
\begin{eqnarray}
\sum_{q<0}\|\dot{\Delta}_{q}\Lambda^{s_{c}-1}[\mathcal{G}(t)\tilde{w}_{0}]\|_{L^2} \lesssim \|\tilde{w}_{0}\|_{\dot{B}^{-s}_{2,\infty}} (1+t)^{-\frac{s+s_{c}-1}{2}}. \label{R-E69}
\end{eqnarray}
Furthermore, by Lemma \ref{lem4.1}, we can obtain
\begin{eqnarray}
&&\sum_{q<0}\|\dot{\Delta}_{q}\Lambda^{s_{c}-1}\tilde{w}\|_{L^2}
\nonumber\\&\leq&\sum_{q<0}
\|\dot{\Delta}_{q}\Lambda^{s_{c}-1}[\mathcal{G}(t)\tilde{w}_{0}]\|_{L^2}
+\int^{t}_{0}\sum_{q<0}\|\dot{\Delta}_{q}\Lambda^{s_{c}-1}[\mathcal{G}(t-\tau)\mathcal{R}(\tau)]\|_{L^2}d\tau.
\nonumber\\&\lesssim &  \|\tilde{w}_{0}\|_{\dot{B}^{-s}_{2,\infty}} (1+t)^{-\frac{s+s_{c}-1}{2}}+
\int^{t}_{0}(1+t-\tau)^{-\frac{s+s_{c}-1}{2}}\|\mathcal{R}(\tau)\|_{\dot{B}^{-s}_{2,\infty}}d\tau, \label{R-E70}
\end{eqnarray}
Just doing the same procedure leading to (\ref{R-E55}), we can obtain (\ref{R-E56}).
\end{proof}

The subsequent lemma is related to the nonlinear high-frequency estimates of solutions.
\begin{lem}\label{lem4.3}(High-frequency estimates)
 Under the assumption of Proposition \ref{prop4.2}, we have
\begin{eqnarray}
\sum_{q\geq0}2^{q(s_{c}-1-\ell)}\|\Delta_{q}\Lambda^{\ell}\tilde{w}\|_{L^2}\lesssim \|\tilde{w}_{0}\|_{B^{s_{c}}_{2,1}}e^{-c_{1}t}
+(1+t)^{-\frac{s+\ell}{2}}\mathcal{E}_{0}(t)\mathcal{E}_{1}(t) \label{R-E71}
\end{eqnarray}
for $0\leq\ell\leq s_{c}-1$.
\end{lem}

\begin{proof}
Due to $\Delta_{q}f\equiv\dot{\Delta}_{q}f(q\geq0$), it suffices to show (\ref{R-E71}) for the inhomogeneous case.
From (\ref{R-E37}), we get
\begin{eqnarray}
\|\Delta_{q}\Lambda^{\ell}\mathcal{G}(t)\tilde{w}_{0}\|_{L^2}\lesssim e^{-c_{1}t}\|\Delta_{q}\Lambda^{\ell}\tilde{w}_{0}\|_{L^2} \label{R-E72}
\end{eqnarray}
for all $q\geq0$. It follows from Lemma \ref{lem4.1} that
\begin{eqnarray}
&&\|\Delta_{q}\Lambda^{\ell}\tilde{w}\|_{L^2}
\nonumber\\&\leq&\|\Delta_{q}\Lambda^{\ell}[\mathcal{G}(t)\tilde{w}_{0}]\|_{L^2}+\int^{t}_{0}\|\Delta_{q}\Lambda^{\ell}[\mathcal{G}(t-\tau)\mathcal{R}(\tau)]\|_{L^2}d\tau
\nonumber\\&\lesssim& e^{-c_{2}t}\|\Delta_{q}\Lambda^{\ell}\tilde{w}_{0}\|_{L^2}+\int^{t}_{0}e^{-c_{2}(t-\tau)}\|\Delta_{q}\Lambda^{\ell}\mathcal{R}(\tau)\|_{L^2}d\tau
\nonumber\\&\lesssim& e^{-c_{2}t}\|\Delta_{q}\Lambda^{\ell}\tilde{w}_{0}\|_{L^2}+\|\Delta_{q}\Lambda^{\ell}\mathcal{R}(t)\|_{L^2} \label{R-E73}
\end{eqnarray}
which leads to
\begin{eqnarray}
\sum_{q\geq0}2^{q(s_{c}-1-\ell)}\|\Delta_{q}\Lambda^{\ell}\tilde{w}\|_{L^2}\lesssim \|\tilde{w}_{0}\|_{B^{s_{c}}_{2,1}}e^{-c_{1}t}+\|\Lambda^{\ell}\mathcal{R}(t)\|_{\dot{B}^{s_{c}-1-\ell}_{2,1}} \label{R-E74}
\end{eqnarray}
for $0\leq\ell\leq s_{c}-1$. Then, what left is to estimates the norm $\|\Lambda^{\ell}\mathcal{R}(t)\|_{\dot{B}^{s_{c}-1-\ell}_{2,1}}$.

For instance, it follows from Proposition \ref{prop2.1} that
\begin{eqnarray}
&&\|\Lambda^{\ell}f_{3}(t)\|_{\dot{B}^{s_{c}-1-\ell}_{2,1}}\nonumber\\&\lesssim& \|-\nabla\Delta^{-1}\mathrm{div}(\sigma_{e}u_{e}-\sigma_{i}u_{i})\|_{\dot{B}^{s_{c}-1}_{2,1}}\nonumber\\&\lesssim&
\|\sigma_{e}u_{e}-\sigma_{i}u_{i}\|_{\dot{B}^{s_{c}-1}_{2,1}}
\nonumber\\&\lesssim& \|\Lambda^{\ell}\sigma_{e}\|_{\dot{B}^{s_{c}-1-\ell}_{2,1}}\|u_{e}\|_{\dot{B}^{s_{c}-1}_{2,1}}
+\|\Lambda^{\ell}\sigma_{i}\|_{\dot{B}^{s_{c}-1-\ell}_{2,1}}\|u_{i}\|_{\dot{B}^{s_{c}-1}_{2,1}}
\nonumber\\&\lesssim&
\begin{cases}
\|\Lambda^{\ell}\sigma_{e}\|_{B^{s_{c}-1-\ell}_{2,1}}\|u_{e}\|_{B^{s_{c}}_{2,1}}+ \|\Lambda^{\ell}\sigma_{ei}\|_{B^{s_{c}-1-\ell}_{2,1}}\|u_{i}\|_{B^{s_{c}}_{2,1}}, \ \ 0\leq\ell<s_{c}-1;\\
\|\Lambda^{s_{c}-1}\sigma_{e}\|_{\dot{B}^{0}_{2,1}}\|u_{e}\|_{B^{s_{c}}_{2,1}}+\|\Lambda^{s_{c}-1}\sigma_{i}\|_{\dot{B}^{0}_{2,1}}\|u_{i}\|_{B^{s_{c}}_{2,1}}, \ \ \ell=s_{c}-1;
 \end{cases}
\nonumber\\&\lesssim& (1+t)^{-\frac{s+\ell}{2}}\mathcal{E}_{0}(t)\mathcal{E}_{1}(t).  \label{R-E75}
\end{eqnarray}
Similarly,
\begin{eqnarray}
\|\Lambda^{\ell}f_{ja}(t)\|_{\dot{B}^{s_{c}-1-\ell}_{2,1}}\lesssim (1+t)^{-\frac{s+\ell}{2}}\mathcal{E}_{0}(t)\mathcal{E}_{1}(t), \label{R-E76}
\end{eqnarray}
where $j=1,2$ and $a=e,i$.  Therefore, we obtain
\begin{eqnarray}
\|\Lambda^{\ell}\mathcal{R}(t)\|_{\dot{B}^{s_{c}-1-\ell}_{2,1}}\lesssim (1+t)^{-\frac{s+\ell}{2}}\mathcal{E}_{0}(t)\mathcal{E}_{1}(t). \label{R-E77}
\end{eqnarray}
Finally, together with (\ref{R-E74}) and (\ref{R-E77}), the inequality (\ref{R-E71}) is followed.
\end{proof}

To obtain the sharp decay estimates for velocities, we meet with the difficulty arising from the coupled electric field
$E$, since it has no addition half decay rate. Here, new observations on the information behind the two-fluid Euler-Poisson equations enable us to overcome it.
Firstly, we give the time-weighted estimates for the sum of two velocities.

\begin{lem} \label{lem4.4} (Estimates for the sum of two velocities)
 Under the assumption of Proposition \ref{prop4.2}, we have
\begin{eqnarray}
&&\|\Lambda^{\ell}(u_{e}+u_{i})(t)\|_{B^{s_{c}-2-\ell}_{2,1}}\nonumber\\&\lesssim&
e^{-t}\|(u_{e0},u_{i0}))\|_{B^{s_{c}-2}_{2,1}}+(1+t)^{-\frac{s+\ell+1}{2}}\mathcal{E}_{1}(t)
\nonumber\\&&+(1+t)^{-\frac{s+\ell+1}{2}}\mathcal{E}_{0}(t)\mathcal{E}_{1}(t)+(1+t)^{-s-\frac{\ell+1}{2}}\mathcal{E}_{1}^2(t) \label{R-E78}
\end{eqnarray}
for $0\leq\ell< s_{c}-2$;
\begin{eqnarray}
&&\|\Lambda^{s_{c}-2}(u_{e}+u_{i})(t)\|_{\dot{B}^{0}_{2,1}}\nonumber\\&\lesssim&
e^{-t}\|(u_{e0},u_{i0}))\|_{B^{s_{c}-2}_{2,1}}+(1+t)^{-\frac{s+s_{c}-1}{2}}\mathcal{E}_{1}(t)
\nonumber\\&&+(1+t)^{-\frac{s+s_{c}-1}{2}}\mathcal{E}_{0}(t)\mathcal{E}_{1}(t). \label{R-E79}
\end{eqnarray}
\end{lem}

\begin{proof}
We rewrite the second and fourth equations of (\ref{R-E48}) as
\begin{equation}
\left\{
\begin{array}{l}\partial_{t}u_{e}+u_{e}+\nabla \sigma_{e}-E=f_{2e},\\
\partial_{t}u_{i}+u_{i}+\nabla \sigma_{i}+E=f_{2i}.\\
 \end{array} \right. \label{R-E80}
\end{equation}
By adding two equations in (\ref{R-E80}) to eliminate $E$, and then applying the operator $\Delta_{q}\Lambda^{\ell}(q\geq-1, \ 0\leq\ell\leq s_{c}-2)$ to the resulting equality, we arrive at
\begin{equation}
\partial_{t}\Delta_{q}\Lambda^{\ell}(u_{e}+u_{i})+\Delta_{q}\Lambda^{\ell}(u_{e}+u_{i})=-\Delta_{q}\Lambda^{\ell}(\nabla \sigma_{e}+\nabla \sigma_{i})
+\Delta_{q}\Lambda^{\ell}(f_{2e}+f_{2i}). \label{R-E81}
\end{equation}
Solving the ordinary equation and taking the $L^2$-norm gives
\begin{eqnarray}
&&\|\Delta_{q}\Lambda^{\ell}(u_{e}+u_{i})(t)\|_{L^2}\nonumber\\&\lesssim &\|\Delta_{q}\Lambda^{\ell}(u_{e0},u_{i0})\|_{L^2}e^{-t}
+\int^{t}_{0}e^{-(t-\tau)}\Big(\|\Delta_{q}\Lambda^{\ell}(\nabla \sigma_{e},\nabla \sigma_{i})\|_{L^2}\nonumber\\&&+\|\Delta_{q}\Lambda^{\ell}(f_{2e},f_{2i})\|_{L^2}\Big)d\tau, \label{R-E82}
\end{eqnarray}
In case that $0\leq\ell<s_{c}-2$, by multiplying the factor $2^{q(s_{c}-2-\ell)}$ on both sides of (\ref{R-E82}) and summing up the resulting inequality, we are led to
\begin{eqnarray}
&&\|\Lambda^{\ell}(u_{e},u_{i})(t)\|_{B^{s_{c}-2-\ell}_{2,1}}\nonumber\\&\leq&
e^{-t}\|(u_{e0},u_{i0}))\|_{B^{s_{c}-2}_{2,1}}+\int^{t}_{0}e^{-(t-\tau)}
\Big(\|(\Lambda^{\ell}\nabla \sigma_{e},\Lambda^{\ell}\nabla \sigma_{i})\|_{B^{s_{c}-2-\ell}_{2,1}}
\nonumber\\&&+\|\Lambda^{\ell}(f_{2e},f_{2i})\|_{B^{s_{c}-2-\ell}_{2,1}}\Big)d\tau, \label{R-E83}
\end{eqnarray}
where the linear terms can be estimated as
\begin{eqnarray}
&&\|(\Lambda^{\ell}\nabla \sigma_{e},\Lambda^{\ell}\nabla \sigma_{i})\|_{B^{s_{c}-2-\ell}_{2,1}}\nonumber\\&\lesssim& \|(\Lambda^{\ell+1}\sigma_{e},\Lambda^{\ell+1} \sigma_{i})\|_{B^{s_{c}-1-(\ell+1)}_{2,1}}\lesssim  (1+\tau)^{-\frac{s+\ell+1}{2}}\mathcal{E}_{1}(t). \label{R-E84}
\end{eqnarray}

Next, we turn to estimate the nonlinear terms with respect to  $f_{2e}$ and $f_{2i}$.
The norm is decomposed into two parts according to the relation between homogeneous spaces and inhomogeneous spaces in Lemma \ref{lem2.2}. For instance,
by Lemma \ref{lem2.1}, we further get $\|\Lambda^{\ell}f_{2e}\|_{B^{s_{c}-2-\ell}_{2,1}}\lesssim \|f_{2e}\|_{B^{s_{c}-2}_{2,1}}:=\|f_{2e}\|_{\dot{B}^{s_{c}-2}_{2,1}}+\|f_{2e}\|_{L^2}$,
where
\begin{eqnarray}
&&\|f_{2e}\|_{\dot{B}^{s_{c}-2}_{2,1}}\nonumber\\&\leq& \|u_{e}\|_{\dot{B}^{s_{c}-1}_{2,1}}\|\nabla u_{e}\|_{\dot{B}^{s_{c}-2}_{2,1}}+\|h(\sigma_{e})\|_{\dot{B}^{s_{c}-1}_{2,1}}\|\nabla \sigma_{e}\|_{\dot{B}^{s_{c}-2}_{2,1}}
\nonumber\\&\lesssim& \|u_{e}\|_{B^{s_{c}}_{2,1}}\|\Lambda^{\ell+1}u_{e}\|_{B^{s_{c}-1-(\ell+1)}_{2,1}}
+\|\sigma_{e}\|_{B^{s_{c}}_{2,1}}\|\Lambda^{\ell+1}\sigma_{e}\|_{B^{s_{c}-1-(\ell+1)}_{2,1}}
\nonumber\\&\leq& (1+\tau)^{-\frac{s+\ell+1}{2}}\mathcal{E}_{0}(t)\mathcal{E}_{1}(t) \label{R-E85}
\end{eqnarray}
and
\begin{eqnarray}
&&\|f_{2e}\|_{L^2}\nonumber\\&\leq&\|u_{e}\|_{L^\infty}\|\nabla u_{e}\|_{L^2}+\|h(\sigma_{e})\|_{L^\infty}\|\nabla \sigma_{e}\|_{L^2}
\nonumber\\&\lesssim& \|\Lambda^{\ell}u_{e}\|_{\dot{B}^{s_{c}-1-\ell}_{2,1}}\|\nabla u_{e}\|_{B^{s_{c}-2}_{2,1}}+\|\Lambda^{\ell}\sigma_{e}\|_{\dot{B}^{s_{c}-1-\ell}_{2,1}}\|\nabla \sigma_{e}\|_{B^{s_{c}-2}_{2,1}}
\nonumber\\&\leq& (1+\tau)^{-s-\frac{\ell+1}{2}}\mathcal{E}_{1}^2(t). \label{R-E86}
\end{eqnarray}
Note that Lemmas \ref{lem2.1}-\ref{lem2.2} and Proposition \ref{prop2.2} have been used in (\ref{R-E85})-(\ref{R-E86}).
Furthermore, it follows from (\ref{R-E85})-(\ref{R-E86}) that
\begin{eqnarray}
\|\Lambda^{\ell}f_{2e}\|_{B^{s_{c}-2-\ell}_{2,1}}\lesssim (1+\tau)^{-\frac{s+\ell+1}{2}}\mathcal{E}_{0}(t)\mathcal{E}_{1}(t)+(1+\tau)^{-s-\frac{\ell+1}{2}}\mathcal{E}_{1}^2(t). \label{R-E87}
\end{eqnarray}
Similarly,
\begin{eqnarray}
\|\Lambda^{\ell}f_{2i}\|_{B^{s_{c}-2-\ell}_{2,1}}\lesssim (1+\tau)^{-\frac{s+\ell+1}{2}}\mathcal{E}_{0}(t)\mathcal{E}_{1}(t)+(1+\tau)^{-s-\frac{\ell+1}{2}}\mathcal{E}_{1}^2(t). \label{R-E88}
\end{eqnarray}
Finally, by combining inequalities (\ref{R-E83})-(\ref{R-E84}) and (\ref{R-E87})-(\ref{R-E88}), we obtain (\ref{R-E78}) directly.

On the other hand, in case that $\ell=s_{c}-2$, by applying the operator $\dot{\Delta}_{q}\Lambda^{\sigma_{c}-2}(q\in \mathbb{Z})$ to
the sum of two velocity equations and performing the similar procedure leading to (\ref{R-E83}), we obtain
\begin{eqnarray}
&&\|\Lambda^{s_{c}-2}(u_{e},u_{i})(t)\|_{\dot{B}^{0}_{2,1}}\nonumber\\&\leq&
e^{-t}\|(u_{e0},u_{i0}))\|_{B^{s_{c}-2}_{2,1}}+\int^{t}_{0}e^{-(t-\tau)}
\Big(\|(\Lambda^{s_{c}-2}(\nabla \sigma_{e},\nabla \sigma_{i})\|_{\dot{B}^{0}_{2,1}}
\nonumber\\&&+\|\Lambda^{s_{c}-2}(f_{2e},f_{2i})\|_{\dot{B}^{0}_{2,1}}\Big)d\tau. \label{R-E89}
\end{eqnarray}
Next, we revise the inequalities (\ref{R-E84})-(\ref{R-E86}) a little as follows:
\begin{eqnarray}
\|\Lambda^{s_{c}-2}(\nabla \sigma_{e},\nabla \sigma_{i})\|_{\dot{B}^{0}_{2,1}}\approx\|\Lambda^{s_{c}-1}(\sigma_{e},\sigma_{i})\|_{\dot{B}^{0}_{2,1}}
\lesssim  (1+\tau)^{-\frac{s+s_{c}-1}{2}}\mathcal{E}_{1}(t), \label{R-E90}
\end{eqnarray}
\begin{eqnarray}
\|\Lambda^{s_{c}-2}f_{2e}\|_{\dot{B}^{0}_{2,1}}&\approx&\|f_{2e}\|_{\dot{B}^{s_{c}-2}_{2,1}}
\nonumber\\&\lesssim& \|u_{e}\|_{\dot{B}^{s_{c}-1}_{2,1}}
\|\nabla u_{e}\|_{\dot{B}^{s_{c}-2}_{2,1}}+\|h(\sigma_{e})\|_{\dot{B}^{s_{c}-1}_{2,1}}\|\nabla \sigma_{e}\|_{\dot{B}^{s_{c}-2}_{2,1}}
\nonumber\\&\lesssim& \|u_{e}\|_{B^{s_{c}}_{2,1}}\|\Lambda^{s_{c}-1}u_{e}\|_{\dot{B}^{0}_{2,1}}
+\|\sigma_{e}\|_{B^{s_{c}}_{2,1}}\|\Lambda^{s_{c}-1}\sigma_{e}\|_{\dot{B}^{0}_{2,1}}
\nonumber\\&\leq& (1+\tau)^{-\frac{s+s_{c}-1}{2}}\mathcal{E}_{0}(t)\mathcal{E}_{1}(t)\label{R-E91}
\end{eqnarray}
and
\begin{eqnarray}
\|\Lambda^{s_{c}-2}f_{2i}\|_{\dot{B}^{0}_{2,1}}\leq (1+\tau)^{-\frac{s+s_{c}-1}{2}}\mathcal{E}_{0}(t)\mathcal{E}_{1}(t). \label{R-E92}
\end{eqnarray}
Hence, the inequality (\ref{R-E79}) is followed by (\ref{R-E89})-(\ref{R-E92}).
\end{proof}

Secondly, we shall prove the sharp time-weighted estimates for the difference of two velocities. Indeed, our idea of the proof is from the new observation that the linearized part of difference equations for velocities, densities along with the electron field $E$ exactly consist with a one-fluid Euler-Poisson equation. Therefore, we obtain the exponential decay for the linearized solutions. Furthermore, the desired decay estimate of the difference of velocities is shown by the frequency-localization Duhamel principle.

\begin{lem}\label{lem4.5}(Estimates for the difference of two velocities) Under the assumption of Proposition \ref{prop4.2}, we have
\begin{eqnarray}
&&\|\Lambda^{\ell}(u_{e}-u_{i})(t)\|_{B^{s_{c}-2-\ell}_{2,1}}\nonumber\\&\lesssim&
e^{-c_{3}t}\|\tilde{w}_{0}\|_{B^{s_{c}-2}_{2,1}}
+(1+t)^{-\frac{s+\ell+1}{2}}\mathcal{E}_{0}(t)\mathcal{E}(t)\nonumber\\&&+(1+t)^{-s-\frac{\ell+1}{2}}\mathcal{E}_{1}(t)\mathcal{E}(t) \label{R-E93}
\end{eqnarray}
for $0\leq\ell< s_{c}-2$;
\begin{eqnarray}
\|\Lambda^{s_{c}-2}(u_{e}-u_{i})(t)\|_{\dot{B}^{0}_{2,1}}\lesssim
e^{-c_{3}t}\|\tilde{w}_{0}\|_{B^{s_{c}-2}_{2,1}}+(1+t)^{-\frac{s+s_{c}-1}{2}}\mathcal{E}_{0}(t)\mathcal{E}(t), \label{R-E94}
\end{eqnarray}
where $c_{3}>0$ is some constant.
\end{lem}

\begin{proof}
Set $$\tilde{u}=u_{e}-u_{i},\ \ \tilde{\sigma}=\sigma_{e}-\sigma_{i}. $$
Then, it follows from (\ref{R-E48}) that
\begin{equation}
\left\{
\begin{array}{l}
\partial_{t} \tilde{\sigma}+\mathrm{div}\tilde{u}=\tilde{f}_{1},\\
\partial_{t} \tilde{u}+\nabla\tilde{\sigma}+\tilde{u}=2E+\tilde{f}_{2},\\
\partial_{t} E=-\nabla\Delta^{-1}\mathrm{div}\tilde{u}+f_{3}
 \end{array}
\right. \label{R-E95}
\end{equation}
with the initial data
\begin{eqnarray}
(\tilde{\sigma}_{0},\tilde{u}_{0},E_{0})=(\sigma_{e0}-\sigma_{i0}, u_{e0}-u_{i0},\nabla\Delta^{-1}(\sigma_{e0}-\sigma_{i0})), \label{R-E955}
\end{eqnarray}
where
$$\tilde{f}_{1}=-\mathrm{div}\tilde{F}_{1}\ \ \ \mbox{with}\ \ \ \tilde{F}_{1}:=\sigma_{e}u_{e}-\sigma_{i}u_{i},
\ \ \ \tilde{f}_{2}=f_{2e}-f_{2i}.$$

To obtain (\ref{R-E93})-(\ref{R-E94}), we first prove the following claim.

\begin{claim} \label{claim4.1}
The solution of the linearized part of (\ref{R-E95}):
\begin{equation}
\left\{
\begin{array}{l}
\partial_{t} \tilde{\sigma}+\mathrm{div}\tilde{u}=0,\\
\partial_{t} \tilde{u}+\nabla\tilde{\sigma}+\tilde{u}=2E,\\
\partial_{t} E=-\nabla\Delta^{-1}\mathrm{div}\tilde{u},
 \end{array}
\right. \label{R-E951}
\end{equation}
decays exponentially:
\begin{eqnarray}
\|\Lambda^{\ell}(\tilde{\sigma},\tilde{u},E)(t)\|_{B^{s_{c}-2-\ell}_{2,1}}\lesssim e^{-c_{3}t}\|\Lambda^{\ell}(\tilde{\sigma}_{0},\tilde{u}_{0},E_{0})\|_{B^{s_{c}-2-\ell}_{2,1}} \label{R-E9555}
\end{eqnarray}
for $0\leq\ell< s_{c}-2$;
\begin{eqnarray}
\|(\tilde{\sigma},\tilde{u},E)(t)\|_{\dot{B}^{s_{c}-2}_{2,1}}\lesssim e^{-c_{3}t}\|(\tilde{\sigma}_{0},\tilde{u}_{0},E_{0})\|_{\dot{B}^{s_{c}-2}_{2,1}}.\label{R-E9556}
\end{eqnarray}
\end{claim}
Indeed, it suffices to show (\ref{R-E9555}), since (\ref{R-E9556}) can be dealt with the similar manner.
We note some useful equalities:
\begin{eqnarray}
\mathrm{div}\tilde{u}=-\mathrm{div}E_{t},\ \ \ \ \mathrm{div}E=\tilde{\sigma}. \label{R-E96}
\end{eqnarray}
The proof of (\ref{R-E9555}) is to capture the dissipation rates from contributions
of $(\tilde{\sigma},\tilde{u},E)$ in turn by using the low-frequency and high-frequency decomposition
methods.

\underline{\textit{(a) Estimate for the dissipation from $\tilde{u}$}}

Applying the operator $\Delta_{q}\Lambda^{\ell}(q\geq-1, \ 0\leq\ell<s_{c}-2)$ to the first two equations of (\ref{R-E951}) gives
\begin{equation}
\left\{
\begin{array}{l}
\partial_{t} \Delta_{q}\Lambda^{\ell}\tilde{\sigma}+\mathrm{div}\Delta_{q}\Lambda^{\ell}\tilde{u}=0,\\
\partial_{t} \Delta_{q}\Lambda^{\ell}\tilde{u}+\Delta_{q}\Lambda^{\ell}\nabla\tilde{\sigma}+\Delta_{q}\Lambda^{\ell}\tilde{u}=2\Delta_{q}\Lambda^{\ell}E.\\
 \end{array}
\right. \label{R-E97}
\end{equation}
Multiplying the first equation of (\ref{R-E97}) by $\Delta_{q}\Lambda^{\ell}\tilde{\sigma}$, the second one by $\Delta_{q}\Lambda^{\ell}\tilde{u}$ and
adding the resulting equations together, then integrating it over $\mathbb{R}^{3}$, we obtain
\begin{eqnarray}
&&\frac{1}{2}\frac{d}{dt}(\|\Delta_{q}\Lambda^{\ell}\tilde{\sigma}\|^2_{L^2}+\|\Delta_{q}\Lambda^{\ell}\tilde{u}\|^2_{L^2})+\|\Delta_{q}\Lambda^{\ell}\tilde{u}\|^2_{L^2}\nonumber\\
&=&2\int_{\mathbb{R}^{3}}\Delta_{q}\Lambda^{\ell}E\cdot\Delta_{q}\Lambda^{\ell}\tilde{u}, \label{R-E98}
\end{eqnarray}
where the electronic field term can be estimated by (\ref{R-E96}):
\begin{eqnarray}
&&\int_{\mathbb{R}^{3}}\Delta_{q}\Lambda^{\ell}E\cdot\Delta_{q}\Lambda^{\ell}\tilde{u}\nonumber\\&=&
-\int_{\mathbb{R}^{3}}\Delta_{q}\Lambda^{\ell}\mathit\Phi\Delta_{q}\Lambda^{\ell}\mathrm{div}\tilde{u}
\nonumber\\&=&\int_{\mathbb{R}^{3}}\Delta_{q}\Lambda^{\ell}\mathit\Phi\Delta_{q}\Lambda^{\ell}\mathrm{div}E_{t}
\nonumber\\&=&-\frac{1}{2}\frac{d}{dt}\|\Delta_{q}\Lambda^{\ell}E\|^2_{L^2}. \label{R-E99}
\end{eqnarray}
Then, combining (\ref{R-E98})-(\ref{R-E99}) gives
\begin{eqnarray}
&&\hspace{-10mm}\frac{1}{2}\frac{d}{dt}(\|\Delta_{q}\Lambda^{\ell}\tilde{\sigma}\|^2_{L^2}+\|\Delta_{q}\Lambda^{\ell}\tilde{u}\|^2_{L^2}
+2\|\Delta_{q}\Lambda^{\ell}E\|^2_{L^2})+\|\Delta_{q}\Lambda^{\ell}\tilde{u}\|^2_{L^2}\leq0. \label{R-E100}
\end{eqnarray}

\underline{\textit{(b) Estimate for the dissipation from $\tilde{\sigma}$}}

 To do this, we rewrite the second equation of (\ref{R-E951}) as follows:
\begin{equation}
\nabla\tilde{\sigma}=-(\partial_{t}\tilde{u}+\tilde{u}-2E). \label{R-E101}
\end{equation}
Then applying the operator $\Delta_{q}\Lambda^{\ell}$ to (\ref{R-E101}) and integrating
it over $\mathbb{R}^3$ after multiplying $\Delta_{q}\Lambda^{\ell}\nabla\tilde{\sigma}$, we have
\begin{eqnarray}
\|\Delta_{q}\Lambda^{\ell}\nabla\tilde{\sigma}\|^2_{L^2}
&=&-\int_{\mathbb{R}^{3}}\partial_{t}\Delta_{q}\Lambda^{\ell}\tilde{u}\cdot\Delta_{q}\Lambda^{\ell}\nabla\tilde{\sigma}
+2\int_{\mathbb{R}^{3}}\Delta_{q}\Lambda^{\ell}E\cdot\Delta_{q}\Lambda^{\ell}\nabla\tilde{\sigma}
\nonumber\\&&-\int_{\mathbb{R}^{3}}\Delta_{q}\Lambda^{\ell}\tilde{u}\cdot\Delta_{q}\Lambda^{\ell}\nabla\tilde{\sigma}. \label{R-E102}
\end{eqnarray}
Note that (\ref{R-E96}), integration by parts gives

\begin{eqnarray}
&&2\int_{\mathbb{R}^{3}}\Delta_{q}\Lambda^{\ell}E\cdot\Delta_{q}\Lambda^{\ell}\nabla\tilde{\sigma}\nonumber\\&=&
-2\int_{\mathbb{R}^{3}}\Delta_{q}\Lambda^{\ell}\mathrm{div}E\cdot\Delta_{q}\Lambda^{\ell}\tilde{\sigma}
\nonumber\\&=&-2\|\Delta_{q}\Lambda^{\ell}\tilde{\sigma}\|^2_{L^2} \label{R-E103}
\end{eqnarray}
and
\begin{eqnarray}
&&-\int_{\mathbb{R}^{3}}\partial_{t}\Delta_{q}\Lambda^{\ell}\tilde{u}\cdot\Delta_{q}\Lambda^{\ell}\nabla\tilde{\sigma}
\nonumber\\&=&\int_{\mathbb{R}^{3}}\Delta_{q}\Lambda^{\ell}\mathrm{div}\tilde{u}_{t}\Delta_{q}\Lambda^{\ell}\tilde{\sigma}
\nonumber\\&=&\frac{d}{dt}\int_{\mathbb{R}^{3}}\Delta_{q}\Lambda^{\ell}\mathrm{div}\tilde{u}\Delta_{q}\Lambda^{\ell}\tilde{\sigma}-
\int_{\mathbb{R}^{3}}\Delta_{q}\Lambda^{\ell}\mathrm{div}\tilde{u}\Delta_{q}\Lambda^{\ell}\tilde{\sigma}_{t}
\nonumber\\&=&\frac{d}{dt}\int_{\mathbb{R}^{3}}\Delta_{q}\Lambda^{\ell}\mathrm{div}\tilde{u}\Delta_{q}\Lambda^{\ell}\tilde{\sigma}
+\|\Delta_{q}\Lambda^{\ell}\mathrm{div}\tilde{u}\|^2_{L^2}. \label{R-E104}
\end{eqnarray}
Substituting (\ref{R-E103})-(\ref{R-E104}) into (\ref{R-E102}), with the aid of H\"{o}lder inequality, we further get
\begin{eqnarray}
&&-\frac{d}{dt}\int_{\mathbb{R}^{3}}\Delta_{q}\Lambda^{\ell}\mathrm{div}\tilde{u}\Delta_{q}\Lambda^{\ell}\tilde{\sigma}+2\|\Delta_{q}\Lambda^{\ell}\tilde{\sigma}\|^2_{L^2}+\|\Delta_{q}\Lambda^{\ell}\nabla\tilde{\sigma}\|^2_{L^2}
\nonumber\\&\leq&\|\Delta_{q}\Lambda^{\ell}\mathrm{div}\tilde{u}\|^2_{L^2}
+\|\Delta_{q}\Lambda^{\ell}\tilde{u}\|_{L^2}\|\Delta_{q}\Lambda^{\ell}\nabla\tilde{\sigma}\|_{L^2}. \label{R-E105}
\end{eqnarray}

By using Lemma \ref{lem2.1}, we are led to the high-frequency estimate and low-frequency one, respectively:
\begin{eqnarray}
-\frac{d}{dt}\int_{\mathbb{R}^{3}}\Delta_{q}\Lambda^{\ell}\mathrm{div}\tilde{u}\Delta_{q}\Lambda^{\ell}\tilde{\sigma}+2^{2q}\|\Delta_{q}\Lambda^{\ell}\tilde{\sigma}\|^2_{L^2}
\lesssim2^{2q}\|\Delta_{q}\Lambda^{\ell}\tilde{u}\|^2_{L^2}\ \  (q\geq0). \label{R-E106}
\end{eqnarray}
and
\begin{eqnarray}
-\frac{d}{dt}\int_{\mathbb{R}^{3}}\Delta_{-1}\Lambda^{\ell}\mathrm{div}\tilde{u}\Delta_{-1}\Lambda^{\ell}\tilde{\sigma}+\|\Delta_{-1}\Lambda^{\ell}\tilde{\sigma}\|^2_{L^2}
\lesssim\|\Delta_{-1}\Lambda^{\ell}\tilde{u}\|^2_{L^2}. \label{R-E107}
\end{eqnarray}

\underline{\textit{(c) Estimate for the dissipation from $E$}}

From the equation $\mathrm{div}E=\tilde{\sigma}$ and the irrotationality of $E$, we can get the high-frequency estimate
\begin{eqnarray}
2^{q}\|\Delta_{q}\Lambda^{\ell}E\|^2_{L^2}\lesssim \|\Delta_{q}\Lambda^{\ell}\tilde{\sigma}\|_{L^2}\|\Delta_{q}\Lambda^{\ell}E\|_{L^2} \ (q\geq0). \label{R-E108}
\end{eqnarray}
For the low-frequency case, we need to perform the different estimate. Precisely,
\begin{eqnarray}
&&-\frac{d}{dt}\int_{\mathbb{R}^{3}}\Delta_{-1}\Lambda^{\ell}E\cdot\Delta_{-1}\Lambda^{\ell}\tilde{u}\nonumber\\
&=&-\int_{\mathbb{R}^{3}}\Delta_{-1}\Lambda^{\ell}E_{t}\cdot\Delta_{-1}\Lambda^{\ell}\tilde{u}-\int_{\mathbb{R}^{3}}\Delta_{-1}\Lambda^{\ell}E\cdot\Delta_{-1}\Lambda^{\ell}\tilde{u}_{t}
\nonumber\\&=&\int_{\mathbb{R}^{3}}\nabla\Delta^{-1}\mathrm{div}\Delta_{-1}\Lambda^{\ell}\tilde{u}\cdot\Delta_{-1}\Lambda^{\ell}\tilde{u}
-2\|\Delta^{-1}\Lambda^{\ell}E\|^2_{L^2}\nonumber\\&&+\int_{\mathbb{R}^{3}}\Delta_{-1}\Lambda^{\ell}E\cdot\Delta_{-1}\Lambda^{\ell}(\tilde{u}+\nabla\tilde{\sigma}), \label{R-E109}
\end{eqnarray}
which leads to
\begin{eqnarray}
&&-\frac{d}{dt}\int_{\mathbb{R}^{3}}\Delta_{-1}\Lambda^{\ell}E\cdot\Delta_{-1}\Lambda^{\ell}\tilde{u}+2\|\Delta^{-1}\Lambda^{\ell}E\|^2_{L^2}
\nonumber\\&\leq&\|\Delta_{-1}\Lambda^{\ell}\tilde{u}\|^2_{L^2}
+(\|\Delta_{-1}\Lambda^{\ell}\tilde{u}\|_{L^2}+\|\Delta_{-1}\Lambda^{\ell}\tilde{\sigma}\|_{L^2})
\|\Delta^{-1}\Lambda^{\ell}E\|_{L^2}. \label{R-E110}
\end{eqnarray}

The next step is to combine above inequalities (\ref{R-E100}) and (\ref{R-E106})-(\ref{R-E108})
and (\ref{R-E110}).  We omit the details for brevity. Furthermore,
we can conclude that there exists  a constant $c_{3}>0$ such that the following differential inequality holds
\begin{eqnarray}
\frac{d}{dt}Q(t)+c_{3}\|\Lambda^{\ell}(\tilde{u},\tilde{\sigma},E)\|_{B^{s_{c}-2-\ell}_{2,1}}\leq 0, \label{R-E111}
\end{eqnarray}
where
$$Q(t)\approx\|\Lambda^{\ell}(\tilde{u},\tilde{\sigma},E)(t,\cdot)\|_{B^{s_{c}-2-\ell}_{2,1}}.$$
The standard Gronwall's inequality implies (\ref{R-E9555}) immediately. Hence the proof of Claim \ref{claim4.1} is complete.

By virtue of the frequency-localization Duhamel principle related to the system (\ref{R-E95})-(\ref{R-E955}), we can arrive at
\begin{eqnarray}
&&\|\Lambda^{\ell}\tilde{u}\|_{B^{s_{c}-2-\ell}_{2,1}}\nonumber\\&\leq&\|\Lambda^{\ell}(\tilde{u},\tilde{\sigma},E)\|_{B^{s_{c}-2-\ell}_{2,1}}\nonumber\\&\lesssim & \hspace{-2mm} e^{-c_{3}t}\|\tilde{w}_{0}\|_{B^{s_{c}-2}_{2,1}}
+\int^{t}_{0}e^{-c_{3}(t-\tau)}
\|\Lambda^{\ell}(\tilde{f}_{1},\tilde{f}_{2},\tilde{F}_{1},f_{3})(\tau)\|_{B^{s_{c}-2-\ell}_{2,1}}d\tau. \label{R-E112}
\end{eqnarray}
Then doing the same procedure leading to (\ref{R-E87})-(\ref{R-E88}) gives (\ref{R-E93}). Additionally,
(\ref{R-E94}) is followed by (\ref{R-E9556}) similarly. Hence, the proof of Lemma \ref{lem4.4} is complete ultimately.
\end{proof}

Having Lemmas \ref{lem4.4}-\ref{lem4.5}, by the elementary formula
$$u_{e}=\frac{1}{2}\Big\{(u_{e}+u_{i})+(u_{e}-u_{i})\Big\},\ \ \  u_{i}=\frac{1}{2}\Big\{(u_{e}+u_{i})-(u_{e}-u_{i})\Big\},$$
we obtain the time-weighted energy estimates for $(u_{e},u_{i})$.

\begin{cor}\label{cor4.1}
Under the assumption of Proposition \ref{prop4.2}, it holds that
\begin{eqnarray}
&&\|\Lambda^{\ell}(u_{e},u_{i})(t)\|_{B^{s_{c}-2-\ell}_{2,1}}\nonumber\\&\lesssim&
e^{-c_{4}t}\|\tilde{w}_{0}\|_{B^{s_{c}-2}_{2,1}}+(1+t)^{-\frac{s+\ell+1}{2}}\mathcal{E}_{1}(t)
+(1+t)^{-\frac{s+\ell+1}{2}}\mathcal{E}_{0}(t)\mathcal{E}(t)\nonumber\\&&+(1+t)^{-s-\frac{\ell+1}{2}}\mathcal{E}_{1}(t)\mathcal{E}(t) \label{R-E113}
\end{eqnarray}
for $0\leq\ell<s_{c}-2$;
\begin{eqnarray}
&&\|\Lambda^{s_{c}-2}(u_{e},u_{i})(t)\|_{\dot{B}^{0}_{2,1}}\nonumber\\ &\lesssim&
\hspace{-3mm} e^{-c_{4}t}\|\tilde{w}_{0}\|_{B^{s_{c}-2}_{2,1}}+(1+t)^{-\frac{s+\ell+1}{2}}\mathcal{E}_{1}(t)+(1+t)^{-\frac{s+s_{c}-1}{2}}\mathcal{E}_{0}(t)\mathcal{E}(t), \label{R-E114}
\end{eqnarray}
where the constant $c_{4}:=\min\{1,c_{3}\}.$
\end{cor}

Therefore, having above preparations, the proofs of Propositions \ref{prop4.1}-\ref{prop4.2} can be finished as follows.\\

\noindent \textbf{\textit{The proofs of Propositions \ref{prop4.1}-\ref{prop4.2}}.}
From Lemmas \ref{lem4.2}-\ref{lem4.3}, we deduce that
\begin{eqnarray}
\mathcal{E}_{1}(t)\lesssim \mathcal{M}_{0}+\mathcal{E}^{2}(t)+\mathcal{E}_{0}(t)\mathcal{E}_{1}(t). \label{R-E115}
\end{eqnarray}
On the other hand, it follows from Corollary \ref{cor4.1} that
\begin{eqnarray}
\mathcal{E}_{2}(t)\lesssim \mathcal{M}_{0}+\mathcal{E}_{1}(t)+\mathcal{E}_{0}(t)\mathcal{E}(t)+\mathcal{E}_{1}(t)\mathcal{E}(t). \label{R-E116}
\end{eqnarray}
Therefore, (\ref{R-E115})-(\ref{R-E116}) leads to (\ref{R-E54}) directly.

Furthermore, it follows from Theorem \ref{thm1.1} that
$\mathcal{E}_{0}(t)\lesssim \|\tilde{w}_{0}\|_{B^{s_{c}}_{2,1}}\lesssim \mathcal{M}_{0}$. Thus, if $\mathcal{M}_{0}$ is sufficient small,
then
\begin{eqnarray}
\mathcal{E}(t)\lesssim \mathcal{M}_{0}+\mathcal{E}^{2}(t), \label{R-E117}
\end{eqnarray}
which can deduce that $\mathcal{E}(t)\lesssim \mathcal{M}_{0}$, provided that $\mathcal{M}_{0}$ is sufficient small.

Finally, it follows from $\mathrm{div}E=n_{e}-n_{i}$ that
\begin{eqnarray}
\|\Lambda^{\ell}(n_{e}-n_{i})\|_{B^{s_{c}-2-\ell}_{2,1}}\leq\|\Lambda^{\ell+1}E\|_{B^{s_{c}-1-(\ell+1)}_{2,1}}\leq (1+t)^{-\frac{s+\ell+1}{2}}\mathcal{E}_{1}(t)\label{R-E118}
\end{eqnarray}
for $0\leq\ell<s_{c}-2$, and
\begin{eqnarray}
\|\Lambda^{(s_{c}-2)}(n_{e}-n_{i})\|_{\dot{B}^{0}_{2,1}}\leq\|\Lambda^{s_{c}-1}E\|_{\dot{B}^{0}_{2,1}}\leq(1+t)^{-\frac{s+s_{c}-1}{2}}\mathcal{E}_{1}(t).
\label{R-E119}
\end{eqnarray}

\section{Appendix}\setcounter{equation}{0}\label{sec:5}
For the convenience of reader, we list interpolation inequalities related to Besov spaces, actually, which
parallel the work \cite{SS}. However, we make some simplicity for use,
since their inequalities are related to the mixed spaces containing the microscopic velocity.

\begin{lem}\label{lem8.1}
Suppose $k\geq0$ and $m,\varrho>0$. Then the following inequality holds
\begin{eqnarray}
\|f\|_{\dot{B}^{k}_{2,1}}\lesssim \|f\|^{\theta}_{\dot{B}^{k+m}_{2,\infty}}\|f\|^{1-\theta}_{\dot{B}^{-\varrho}_{2,\infty}} \ \  \mbox{with}\ \ \ \theta=\frac{\varrho+k}{\varrho+k+m}. \label{R-E120}
\end{eqnarray}
\end{lem}

\begin{lem}\label{lem8.2}
Suppose $k\geq0$ and $m,\varrho>0$.  Then the following inequality holds
\begin{eqnarray}
\|\Lambda^{k}f\|_{L^2}\lesssim \|\Lambda^{k+m}f\|^{\theta}_{L^2}\|f\|^{1-\theta}_{\dot{B}^{-\varrho}_{2,\infty}} \ \  \mbox{with}\ \ \ \theta=\frac{\varrho+k}{\varrho+k+m}, \label{R-E121}
\end{eqnarray}
where (\ref{R-E121}) is also true for $\partial^{\alpha}$ with $|\alpha|=k\ (k$ nonnegative integer).
\end{lem}

\begin{lem}\label{lem8.3}
Suppose that $m\neq\varrho$. Then the following inequality holds
\begin{eqnarray}
\|f\|_{\dot{B}^{k}_{p,1}}\lesssim \|f\|^{1-\theta}_{\dot{B}^{m}_{r,\infty}}\|f\|^{\theta}_{\dot{B}^{\varrho}_{r,\infty}}, \label{R-E122}
\end{eqnarray}
where $0<\theta<1$,  $1\leq r\leq p\leq\infty$ and $$k+n\Big(\frac{1}{r}-\frac{1}{p}\Big)=m(1-\theta)+\varrho\theta.$$
\end{lem}

\begin{lem}\label{lem8.4}
Suppose that $m\neq\varrho$. One has the interpolation inequality of Gagliardo-Nirenberg-Sobolev type
\begin{eqnarray}
\|\Lambda^{k}f\|_{L^{q}}\lesssim \|\Lambda^{m}f\|^{1-\theta}_{L^{r}}\|\Lambda^{\varrho}f\|^{\theta}_{L^{r}}, \label{R-E123}
\end{eqnarray}
where $0\leq\theta\leq1$,  $1\leq r\leq q\leq\infty$ and $$k+n\Big(\frac{1}{r}-\frac{1}{q}\Big)=m(1-\theta)+\varrho\theta.$$
\end{lem}

\begin{lem}\label{lem8.5}
Suppose that $\varrho>0$ and $1\leq p<2$. It holds that
\begin{eqnarray}
\|f\|_{\dot{B}^{-\varrho}_{r,\infty}}\lesssim \|f\|_{L^{p}} \label{R-E124}
\end{eqnarray}
with $1/p-1/r=\varrho/n$. In particular, this holds with $\varrho=n/2, r=2$ and $p=1$.
\end{lem}

\section*{Acknowledgments}
The authors would like to thank Prof. R.-J Duan for his
warm communication. The first author (J. Xu) is partially supported by the Program for New Century Excellent
Talents in University (NCET-13-0857), Special Foundation of China Postdoctoral
Science Foundation (2012T50466) and the NUAA Fundamental
Research Funds (NS2013076). He would like to thank Prof. S. Kawashima
for giving him an opportunity to work at Kyushu University in Japan. The work is also partially supported by
Grant-in-Aid for Scientific Researches (S) 25220702 and (A) 22244009.

\end{document}